\documentclass[12pt,leqno,a4paper]{article}
\usepackage[T1]{fontenc}
\usepackage{color}

\usepackage[pdfborder={0 0 0}]{hyperref}

\usepackage{amssymb}
\usepackage{amsmath}
\usepackage{amsthm}
\usepackage{mathrsfs}
\usepackage{graphicx}

\usepackage[shortlabels]{enumitem}

\newtheorem{theorem}{Theorem}
\numberwithin{theorem}{section}
\newtheorem{proposition}[theorem]{Proposition}
\newtheorem{lemma}[theorem]{Lemma}
\newtheorem{corollary}[theorem]{Corollary}

\theoremstyle{remark}

\theoremstyle{definition}
\newtheorem{remark}[theorem]{Remark}
\newtheorem{definition}[theorem]{Definition}

\numberwithin{equation}{section}

\newcommand\set[1]{\left\{\,#1\,\right\}}		% set
\newcommand\abs[1]{\left|#1\right|}				% modulus
\DeclareMathOperator{\dist}{dist}				% distance
\DeclareMathOperator{\id}{id}					% identity
\DeclareMathOperator{\tr}{tr}					% trace
\DeclareMathOperator{\curl}{curl}				% curl
\DeclareMathOperator{\divv}{div}				% div

\def\N{\mathbb{N}}

\def\R{\mathbb{R}}

\newcommand{\cA}{{\mathcal A}}

\newcommand{\cC}{{\mathcal C}}

\newcommand{\cK}{{\mathcal K}}

\newcommand{\cS}{{\mathcal S}}

\newcommand{\cU}{{\mathcal U}}

\newcommand{\lamax}{\lambda_{\text{max}}}
\newcommand{\lamin}{\lambda_{\text{min}}}						

\usepackage[numbers]{natbib}

\begin{document}

\title{A new approach to the Rayleigh-Taylor instability\thanks{This project has received funding from the European Research Council (ERC) under the European Union's Horizon 2020 research and innovation programme (grant agreement No. 724298-DIFFINCL).}}
\author{Bj\"orn Gebhard \quad J\'ozsef J. Kolumb\'an \quad L\'aszl\'o Sz\'ekelyhidi Jr.}
\date{}
\maketitle

\begin{abstract}
In this article we consider the inhomogeneous incompressible Euler equations describing two fluids with different constant densities under the influence of gravity as a differential inclusion. By considering the relaxation of the constitutive laws we formulate a general criterion for the existence of infinitely many weak solutions which reflect the turbulent mixing of the two fluids. Our criterion can be verified in the case that initially the fluids are at rest and separated by a flat interface with the heavier one being above the lighter one - the classical configuration giving rise to the Rayleigh-Taylor instability. We construct specific examples when the Atwood number is in the ultra high range, for which the zone in which the mixing occurs grows quadratically in time.  
\end{abstract}
%
%{\bf MSC 2010:} Primary: 37J25; Secondary: 37J45, 37N10, 76B47
%
%{\bf Key words:} vortex dynamics; periodic solutions; stability; Floquet multipliers; bifurcation; Poincar\'e section

\section{Introduction}

We study the mixing of two different density perfect incompressible fluids subject to gravity, when the heavier fluid is on top. In this setting an instability known as the Rayleigh-Taylor instability forms on the interface between the fluids which eventually evolves into turbulent mixing. For an overview of the investigation of this phenomenon originating in the work of Rayleigh \cite{Rayleigh} in 1883 we refer to the articles \cite{Abarzhi, Bof, Bof2, Celani_Mazzino}.

The mathematical model (see for example Section 6.4 of \cite{Marchioro_Pulvirenti}) is given by the inhomogeneous incompressible Euler equations
\begin{align}\label{eq:main}
\begin{split}
\partial_t(\rho v)+\text{div }(\rho v\otimes v) + \nabla p &= - \rho g e_n,\\
\text{div } v&=0,\\
\partial_t \rho + \text{div }(\rho v)&=0,
\end{split}
\end{align}
which we consider on a bounded domain $\Omega\subset\R^n$, $n\geq 2$ and a time interval $[0,T)$. Here $\rho:\Omega\times[0,T)\to\mathbb{R}$ denotes the fluid density,
$v:\Omega\times[0,T)\to\mathbb{R}^n$ is the velocity field, respectively
$p:\Omega\times[0,T)\to\mathbb{R}$ is the pressure, $g>0$ is the gravitational constant and $e_n$ is the n-th Euclidean coordinate vector. 
Compared to the homogenous density case, $\rho\equiv 1$, the solvability of the Cauchy problem of \eqref{eq:main} for a general non-constant initial density distribution is more delicate even in the planar case, see Section 6.4 of \cite{Marchioro_Pulvirenti}. Results concerning the local well-posedness have only been obtained under sufficiently strong regularity assumptions on the initial density, see \cite{Danchin,daVega_Valli_I,daVega_Valli_II,Valli_Zaj} and references therein. However, since we are interested in the mixing of two different fluids, our initial data does not fall into the classes considered in  
\cite{Danchin,daVega_Valli_I,daVega_Valli_II,Valli_Zaj}.

More precisely, we consider \eqref{eq:main} together with initial data $v_0:\Omega\rightarrow\R^n$, $\rho_0:\Omega\rightarrow\R$ satisfying
\begin{align}\label{eq:id_comp}
\divv v_0=0\text{ and } \rho_0\in\set{\rho_-,\rho_+}\text{ a.e.}
\end{align}
with two fixed values $\rho_+>\rho_->0$. In fact our main focus lies on the flat unstable initial configuration
\begin{align}\label{eq:initialdata}
v_0\equiv 0\text{ and }\rho_0(x)= \left\{
\begin{array}{ll}
 \rho_+\text{ when }x_n>0,\\
 \rho_-\text{ when }x_n\leq 0,
\end{array} 
\right.
\end{align}
giving rise to the Rayleigh-Taylor instability. The linear stability analysis of the flat interface has already been investigated in the article of Rayleigh \cite{Rayleigh} and for example can also be found in \cite{Bardos}. Regarding the nonlinear analysis, to the best of our knwoledge there has been so far no existence result of mixing solutions for the case of the discontinuous initial data \eqref{eq:initialdata}. 

In the spirit of the results by De Lellis and the 3rd author \cite{DeL-Sz-Annals,DeL-Sz-Adm}, for the homogeneous incompressible Euler equations, we develop a convex integration strategy for the inhomogeneous Euler system to prove the existence of weak solutions for the Cauchy problem \eqref{eq:main}, \eqref{eq:initialdata}. Similarly to other unstable interface problems that have recently been attacked by means of convex integration, like the Kelvin-Helmholtz instability in \cite{Sz-KH} or the Muskat problem for the incompressible porous media equation in \cite{Cordoba,Sz-Muskat}, we can interpret the ``wild'' behaviour of the weak solutions obtained this way as turbulent mixing. More precisely, we prove the existence of solutions with the following properties:

For $\rho_+>\rho_->0$ define the Atwood number 
$
\cA=\frac{\rho_+-\rho_-}{\rho_++\rho_-}
$ and the quadratic functions $c_\pm:\R\rightarrow\R$,
\[
c_+(t)=\frac{\rho_++\rho_-}{2\sqrt{\rho_-}(\sqrt{\rho_+}+\sqrt{\rho_-})}\cA gt^2,\quad c_-(t)=\frac{\rho_++\rho_-}{2\sqrt{\rho_+}(\sqrt{\rho_+}+\sqrt{\rho_-})}\cA gt^2.
\]
Let $T>0$ and $\Omega=(0,1)\times(-c_-(T),c_+(T))\subset\R^2$.

\begin{theorem}\label{thm:teaser_theorem}
Let $\frac{\rho_+}{\rho_-}\geq \left(\frac{4+2\sqrt{10}}{3}\right)^2$.
The initial value problem \eqref{eq:main}, \eqref{eq:initialdata} has infinitely many weak admissible solutions $(\rho,v)\in L^\infty(\Omega\times(0,T);\R\times\R^2)$ with $\rho\in\{\rho_-,\rho_+\}$ a.e. and such that
\begin{enumerate}[(i)]
\item $\rho(x,t)=\rho_+$, $v=0$ for $x_2\geq c_+(t)$,
\item $\rho(x,t)=\rho_-$, $v=0$ for $x_2\leq -c_-(t)$,
\item for any open ball $B$ contained in $\set{(x,t)\in\Omega\times(0,T):x_2\in(-c_-(t),c_+(t))}$ there holds
\[
\int_B\rho_+-\rho(x,t)\:dx\:dt\cdot\int_B\rho(x,t)-\rho_-\:dx\:dt>0.
\]
\end{enumerate}
\end{theorem}
For the precise definition of weak admissible solutions we refer to Definitions \ref{def:weaksols}, \ref{def:admissibility_of_weak_solutions}.

We would like to point out that the infinitely many weak solutions differ only in their turbulent fine structure, while they all have a continuous coarse grained density profile $\bar{\rho}$ in common. The profile $\bar{\rho}(x,t)=\bar{\rho}(x_2,t)$ can be seen as an $x_1$-average of the solutions and is found as the entropy solution of a conservation law
\[
\partial_t\bar{\rho}+gt\partial_{x_2}G(\bar{\rho})=0,
\]
which up to the factor $t$ shows similarities to the conservation law appearing in Otto's relaxation for the incompressible porous media equation \cite{Otto}.
 Further details and the explicit profile $\bar{\rho}$ can be found in Section \ref{sec:subsolutions}.

The condition that the
density ratio $\rho_+/\rho_-$ is larger than $\left(\frac{4+2\sqrt{10}}{3}\right)^2 \approx 11.845 $,
 implies that the
 Atwood number is in the so-called (e.g. \cite{s_ultrahigh}) ``ultra high'' range $(0.845,1)$.
This regime has been of great interest to the physics and numerics communities recently, as it has many applications in fields such as inertial confinement fusion, astrophysics or meteorology (see e.g. \cite{s_ultrahigh,exp_dens,s_all_r}).
For instance the Atwood number for mixing hydrogen and air is $0.85$ (see \cite{s_all_r}).
  
   Higher Atwood number implies higher turbulence, and compared to the low Atwood regime, one can not use the Boussinesq approximation (see e.g. \cite{Bof2,Celani_Mazzino,s_no_b}) to accurately model the phenomena. Compared to the homogeneous density case, where the turbulence is only due to mixing in momentum, here it is due to mixing both in momentum and in density, this ``double mixing'' is reflected also in our relaxation given in Section \ref{sec:statement}.
 
We note that up to our knowledge, our result is the first rigorous result leading to existence of weak solutions with quadratic growth in time for the mixing zone. It is also of interest that both numerical simulations and physical experiments predict a growth rate of the mixing zone like $\alpha\mathcal{A}g t^2$, but there is considerable disagreement about the value of the constant $\alpha$  and its possible dependence on $\mathcal{A}$ (see \cite{s_ultrahigh,exp_dens,s_all_r}).

In future work we plan to further study the possibility of constructing solutions with different mixing zone growth rates, to investigate the optimality of the growth rates $c_\pm$ in Theorem \ref{thm:teaser_theorem}, and to explore more precisely their relation to the values from experiments and simulations.

Concerning convex integration as a tool in the investigation of unstable interface problems we have already  mentioned the papers \cite{Cordoba,Sz-Muskat,Sz-KH}. While \cite{Cordoba} shows the non-uniqueness of solutions to the incompressible porous media equation, the paper \cite{Sz-Muskat} provides the full relaxation of the equation allowing to establish sharp linear bounds for the growth of the mixing zone in the Muskat problem. The knowledge of the relaxation also opened the door to further investigations of the Muskat interface problem, see \cite{Castro_Cordoba_Faraco,Castro_Faraco_Mengual,Foerster_Sz}. We already mentioned the different relaxation approach for the incompressible porous media equation via gradient flow in \cite{Otto}, the unique solution of this relaxation approach turned out to be recovered as a subsolution in \cite{Sz-Muskat}.

Another classical instability in fluid dynamics is the Kelvin-Helmholtz instability generated by vortex sheet initial data. Regarding this instability solutions with linearly growing mixing zone have been constructed in \cite{Sz-KH} based on the computations of the relaxation of the homogeneous Euler equations in \cite{DeL-Sz-Adm}. 

There have also been some recent convex integration results for the compressible Euler \cite{Chio,Feireisl_2,Feireisl} and the inviscid Boussinesq equation \cite{Chiodaroli_Michalek}. The approach used for the compressible Euler equations ultimately relies on reducing the problem to having a finite partition of incompressible and homogeneous fluids. In \cite{Markfelder} the convex hull of the isentropic compressible Euler system has been computed, but so far not used for the construction of weak solutions via convex integration.  In the Boussinesq approximation the influence of density variations is neglected in the left-hand side of the momentum equation \eqref{eq:main}. Moreover, the result in \cite{Chiodaroli_Michalek} addressing the existence of infinitely many weak solutions to a given initial configuration requires the initial density to be of class $\cC^2$ and the obtained weak solutions to this prescribed initial data are not admissible in the sense that they violate the energy inequality. 
We would like to point out that so far there have been no convex integration results relying on the full relaxation of the compressible Euler equations nor the inhomogeneous incompressible Euler equations, the latter will be done in this paper.

The paper is organized as follows. In Section \ref{sec:statement} we present our main results, one regarding the convex integration of the inhomogeneous incompressible Euler equations regardless of initial data, and one regarding the existence  of appropriate subsolutions in the case of a flat initial interface.

In Section \ref{sec:reformulation_as_differential_inclusion} we prove that through an appropriate change of coordinates, which in fact corresponds to the way how actual experiments investigating the Rayleigh-Taylor instability are carried out \cite{exp_dens,Read,exp2}, problem \eqref{eq:main} can be recast as a differential inclusion. The differential inclusion fits in a modified version of the Tartar framework of convex integration, adapted from \cite{DeL-Sz-Adm,Sz-Muskat} to simultaneously handle the absence of the pressure from the set of constraints and the dependence of the set of constraints on $(x,t)$ due to the prescribed energy density function.

In Section \ref{sec:conditions} we prove the ingredients of the topological framework, most importantly we calculate the $\Lambda$-convex hull of the set of constraints, which forms the core of this paper.

In Section \ref{sec:conclusion} we conclude the proof of our main convex integration result, while in Section \ref{sec:subsolutions} we construct appropriate subsolutions having the growth rates presented in Theorem \ref{thm:teaser_theorem}.

\section{Statement of results}\label{sec:statement}

Let $\Omega\subset \R^n$ be a bounded domain and $T>0$. Our notion of solution to equation \eqref{eq:main} on $\Omega\times[0,T)$ is as follows.
\begin{definition}[Weak solutions]\label{def:weaksols}
Let $(\rho_0,v_0)\in L^\infty(\Omega)\times L^2(\Omega;\R^n)$ such that 
\eqref{eq:id_comp} holds a.e. in $\Omega$.
We say that $(\rho,v)\in L^\infty(\Omega\times(0,T))\times L^2(\Omega\times(0,T);\R^n)$ is a weak solution to equation \eqref{eq:main} with initial data $(\rho_0,v_0)$ if for any test functions $\Phi\in C^\infty_c(\Omega\times[0,T);\mathbb{R}^n) $, $\Psi_1\in C^\infty_c(\overline{\Omega}\times[0,T)) $, $\Psi_2\in C^\infty_c(\Omega\times[0,T)) $ such that $\Phi$ is divergence-free, we have
\begin{align*}
\int_0^T\int_{\Omega} \left[\rho v \cdot \partial_t \Phi + \langle\rho v\otimes v ,\nabla\Phi\rangle - g \rho \Phi_n \right] \ dx \ dt +\int_{\Omega}\rho_0(x)v_0 (x)\cdot \Phi(x,0)\ dx=0,\\
\int_0^T\int_\Omega v\cdot\nabla \Psi_1\:dx\:dt=0,\\
\int_0^T\int_{\Omega} \left[\rho \partial_t \Psi_2 + \rho v\cdot\nabla\Psi_2 \right] \ dx \ dt +\int_{\Omega} \rho_0 (x) \Psi_2(x,0)\ dx=0,
\end{align*}
and if $\rho(x,t)\in\set{\rho_-,\rho_+}$ for a.e. $(x,t)\in\Omega\times(0,T)$.
\end{definition}
Note that the definition of $v$ being weakly divergence-free includes the no-flux boundary condition. Moreover, the last condition automatically holds true when we deal with smooth solutions of \eqref{eq:main}, because then the density is transported along the flow associated with $v$, but for weaker notions of solutions this property does not necessarily need to be true, see for example \cite{Modena}. 
Furthermore, given a weak solution, the (in general distributional) pressure $p$ is determined up to a function depending only on time, as in the case of the homogeneous Euler equations, see \cite{Temam}.

As in the homogeneous case, one can associate with a weak solution $(\rho,v)$ an energy density function $E\in L^1(\Omega\times(0,T))$ given by
\[
E(x,t):=\frac{1}{2}\rho(x,t)\abs{v(x,t)}^2+\rho(x,t)gx_n.
\] 
Furthermore, for smooth solutions of \eqref{eq:main} one can show that $t\mapsto\int_{\Omega}E(x,t)\:dx$ is constant. For weak solutions this necessarily does not need to be true. As in the case of the homogeneous Euler equations or hyperbolic conservation laws, in order to not investigate physically irrelevant solutions we require our weak solutions to be admissible with respect to the initial energy.
\begin{definition}[Admissible weak solutions]\label{def:admissibility_of_weak_solutions}
A weak solution $(\rho,v)$ in the sense of Definition \ref{def:weaksols} is called admissible provided it satisfies the weak energy inequality
\[
\int_\Omega E(x,t)\:dx\leq \int_\Omega E(x,0)\:dx\text{ for a.e. }t\in(0,T).
\]
\end{definition}

One main contribution of the present article is the relaxation of equation \eqref{eq:main} viewed as a differential inclusion. For the formulation of the relaxation we need the linear system
\begin{align}\label{eq:linear_system_retransformed}
\begin{split}
\partial_t u+\divv S +\nabla P&=-\rho g e_n,\\
\partial_t\rho+\divv u&=0,\\
\divv v&=0,
\end{split}
\end{align}
considered on $\Omega\times (0,T)$ and with $z=(\rho,v,u,S,P)$ taking values in the space $Z=\R\times\R^n\times\R^n\times \cS_0^{n\times n}\times\R$. Here $\cS^{n\times n}_0$ denotes the space of symmetric $n\times n$ matrices with trace $0$. We will also write $\cS^{n\times n}$ for the space of symmetric matrices, $\id\in \cS^{n\times n}$ for the identity and $\lamax(S),\lamin(S)$ for the maximal, minimal resp., eigenvalue of $S\in\cS^{n\times n}$.

As usual, equations \eqref{eq:linear_system_retransformed} will be complemented by a set of pointwise constraints. Let $e:\Omega\times(0,T)\rightarrow\R_+$ be a given function %satisfying
%\begin{align}\label{eq:energycond} 
%E(x,t)-\rho_\pm g x_n \geq 0
%\end{align}
and define for $(x,t)\in\Omega\times (0,T)$ the sets
\begin{multline}\label{eq:definition_of_nonlinear_constraints}
\cK_{(x,t)}:=\left\{z\in Z: \rho\in\{\rho_-,\rho_+\},~u=\rho v,\right.\\
\left.~\rho v\otimes v-S=\left(e(x,t)-\frac{2}{n}\rho gt e_n\cdot v-\frac{1}{n}\rho g^2t^2\right)\id\right\},
\end{multline}
as well as the sets $\cU_{(x,t)}$ by requiring for $z\in\cU_{(x,t)}$ the following four inequalities to hold
\begin{gather}
\rho_-<\rho<\rho_+,\nonumber\\
\begin{split}\label{eq:conditions_for_being_in_tx_dependent_hull2}
&\frac{\rho_+}{n}\frac{\abs{u-\rho_-v+(\rho-\rho_-)gte_n}^2}{(\rho-\rho_-)^2}<e(x,t),\\
&\frac{\rho_-}{n}\frac{\abs{u-\rho_+v+(\rho-\rho_+)gte_n}^2}{(\rho-\rho_+)^2}<e(x,t),
\end{split}\\
\lamax\left(A(z)\right)<e(x,t)-\frac{2}{n}gt e_n\cdot u-\frac{1}{n}\rho g^2t^2,\label{eq:conditions_for_being_in_tx_dependent_hull}
\end{gather}
where
\[
A(z)=\frac{\rho\rho_-\rho_+v\otimes v-\rho_-\rho_+(u\otimes v+v\otimes u)+(\rho_++\rho_--\rho)u\otimes u}{(\rho_+-\rho)(\rho-\rho_-)}-S.
\]

Note that by the definition of $\cK_{(x,t)}$ in \eqref{eq:definition_of_nonlinear_constraints} every solution of \eqref{eq:linear_system_retransformed} taking values in $\cK_{(x,t)}$ a.e. is a solution to the inhomogeneous Euler equations \eqref{eq:main} with $\rho\in\set{\rho_-,\rho_+}$ and associated energy 
\begin{equation}\label{eq:relation_between_E_and_e}
E=\frac{1}{2}\rho\abs{v}^2+\rho gx_n=\frac{n}{2}e(x,t)-\rho gt e_n\cdot v-\frac{1}{2}\rho g^2t^2+\rho gx_n,
\end{equation}
which is equivalent to saying that
\[
\frac{1}{2}\rho\abs{v+gte_n}^2=\frac{n}{2}e(x,t).
\]
Conversely, if we have a solution $(\rho,v,p)$ of \eqref{eq:main} with $\rho\in\set{\rho_-,\rho_+}$ a.e., we can introduce the variables $u=\rho v$, $S=\rho v\otimes v-\frac{1}{n}\rho\abs{v}^2\id$, $P=p+\frac{1}{n}\rho\abs{v}^2$ 
to see that $z=(\rho,v,u,S,P)$ will satisfy system \eqref{eq:linear_system_retransformed} while pointwise taking values $z(x,t)\in\cK_{(x,t)}$, where $\cK_{(x,t)}$ is defined with respect to the function
\[
e(x,t):=\frac{1}{n}\rho(x,t)\abs{v(x,t)+gte_n}^2.
\]

Since the pressure $P$ does not play a role in the set of constraints $\cK_{(x,t)}$, it is convenient to consider the following projection: 
for $z=(\rho,v,u,S,P)\in Z$ we denote
\begin{align}\label{eq:projection}
\pi(z)=(\rho,v,u,S)\in \R\times\R^n\times\R^n\times \mathcal S_0^{n\times n} .
\end{align}

Using the linear system \eqref{eq:linear_system_retransformed} and the definition of $\cU_{(x,t)}$ we define relaxed solutions to \eqref{eq:main} in the following way.
\begin{definition}[Subsolutions]\label{def:subsolEuler}
Let $e:\Omega\times[0,T)\rightarrow\R_+$ be a bounded function.
We say that $z=(\rho,v,u,S,P):\Omega\times(0,T)\to Z$ is a subsolution of \eqref{eq:main} associated with $e$ and the initial data $(\rho_0,v_0)\in L^\infty(\Omega)\times L^2(\Omega;\R^n)$ satisfying \eqref{eq:id_comp} iff $\pi(z)\in L^\infty(\Omega\times(0,T);\pi(Z))$, $P$ is a distribution, 
$z$ solves \eqref{eq:linear_system_retransformed} in the sense that $v$ is weakly divergence-free (including the weak no-flux boundary condition), 
\begin{align*}
\int_0^T\int_{\Omega} \left[u \cdot \partial_t \Phi + \langle S,\nabla\Phi\rangle - g \rho \Phi_n \right] \ dx \ dt +\int_{\Omega}\rho_0(x)v_0 (x)\cdot \Phi(x,0)\ dx=0,\\
\int_0^T\int_{\Omega} \left[\rho \partial_t \Psi + u\cdot\nabla\Psi \right] \ dx \ dt +\int_{\Omega} \rho_0 (x) \Psi(x,0)\ dx=0,
\end{align*}
for any test functions $\Phi\in C^\infty_c(\Omega\times[0,T);\mathbb{R}^n)$, $\divv \Phi =0$, $\Psi\in C^\infty_c(\Omega\times[0,T))$,   
and if there exists an open set $\mathscr{U}\subset \Omega\times(0,T)$, such that the maps $(x,t)\mapsto \pi(z(x,t))$ and $(x,t)\mapsto e(x,t)$ are continuous on $\mathscr{U}$ with 
\begin{gather*}
z(x,t)\in\cU_{(x,t)},\text{ for all }(x,t)\in\mathscr{U},\\
z(x,t)\in\cK_{(x,t)},\text{ for a.e. }(x,t)\in\Omega\times(0,T)\setminus\mathscr{U}.
\end{gather*}
\end{definition}
We call $\mathscr{U}$ the mixing zone of $z$. Moreover, the subsolution is called admissible provided 
\begin{align}\label{eq:subsolenergy}
E_{sub}(x,t):=\frac{n}{2} e(x,t)- gt e_n\cdot u(x,t)-\frac{1}{2}\rho (x,t)g^2t^2+\rho(x,t) gx_n\end{align}
satisfies
\begin{align}\label{eq:weakadm}
\int_\Omega E_{sub}(x,t)\, dx\leq \int_\Omega E_{sub}(x,0)\,dx \text{ for a.e. }t\in(0,T).
\end{align}

We now can state the following criterion for the existence of infinitely many weak solutions.

\begin{theorem}\label{thm:main2} Let $n=2$ and $e:\Omega\times[0,T)\rightarrow\R_+$ be bounded.
If there exists a subsolution $z$ associated with $e$ in the sense of Definition \ref{def:subsolEuler}, then for the same initial data of the subsolution there exist infinitely many weak solutions in the sense of Definition \ref{def:weaksols}, which coincide almost everywhere on $\Omega\times(0,T)\setminus\mathscr{U}$ with $z$ and whose total energy is given by $E$ defined in \eqref{eq:relation_between_E_and_e}.
The solutions are turbulently mixing on $\mathscr{U}$ in the sense that for any open ball $B\subset\mathscr{U}$ there holds
\begin{equation}\label{eq:mixing_property}
\int_B\rho_+-\rho(x,t)\:d(x,t)\cdot\int_B\rho(x,t)-\rho_-\:d(x,t)>0.
\end{equation}
Among these weak solutions there exists a sequence $\{z_k\}_{k\geq 0}$ such that $\rho_k\rightharpoonup\rho$ in $L^2(\mathscr{U})$.
If in addition $\pi(z)$ is in $\cC^0([0,T];L^2(\Omega;\pi(Z)))$ and satisfies \eqref{eq:weakadm} with strict inequality for every $t\in(0,T]$, 
then infinitely many of the induced weak solutions are admissible in the sense of Definition \ref{def:admissibility_of_weak_solutions}.
\end{theorem}

\begin{remark} 
a) The second to last two statements justify to call $\mathscr{U}$ the mixing zone and to interpret the subsolution density $\rho$ as a kind of coarse-grained or averaged density profile.\\
b)
The result carries over to the three- or higher-dimensional case by constructing suitable potentials analoguosly to \cite{DeL-Sz-Annals}, which is not done here, cf. Lemma \ref{lem:locpw}. The other parts of the proof, for example the computation of the $\Lambda$-convex hull in Section \ref{sec:convex_hull}, are carried out in arbitrary dimensions.\\
c) We will see later that the open set $\cU_{(x,t)}$ is indeed the convex hull of $\cK_{(x,t)}$. In particular we can conclude that
weak limits of solutions are subsolutions in the following sense: Let $(\rho_k,v_k)_{k\in\N}$ be a sequence of essentially bounded weak solutions of \eqref{eq:main} and define as before $u_k:=\rho_kv_k$, $S_k:=\rho_kv_k\otimes v_k-\frac{1}{n}\rho_k\abs{v_k}^2\id$. Assume that $z_k':=(\rho_k,v_k,u_k,S_k)\overset{*}{\rightharpoonup} (\rho,v,u,S)=:z'$ in $L^\infty(\Omega\times(0,T);\R\times\R^n\times\R^n\times \cS_0^{n\times n})$. Assume further that there exists a continuous bounded function $e\in\cC^0(\Omega\times(0,T))$, such that $e_k:=\frac{1}{n}\rho_k\abs{v_k+gte_n}^2\rightarrow e$ in $L^\infty(\Omega\times(0,T))$. Then $z'$ supplemented by a possibly distributional $P$ is a weak solution of the linear system \ref{eq:linear_system_retransformed} with $(z'(x,t),P(x,t))\in \overline{\cU}_{(x,t)}$ for a.e. $(x,t)\in \Omega\times(0,T)$, where $\cU_{(x,t)}$ is defined with respect to the function $e$.
\end{remark}

Our second main result addresses the construction of subsolutions associated with the initial data \eqref{eq:initialdata}. Clearly it only makes sense to consider this initial data on domains satisfying $\Omega\cap (\R^{n-1}\times\{0\})\neq \emptyset$.

\begin{definition}[Rayleigh-Taylor subsolution]\label{def:rt_subsolutions} We call a subsolution $z$ of \eqref{eq:main} a Rayleigh-Taylor subsolution (short RT-subsolution) provided the initial data is given by \eqref{eq:initialdata}
 and the subsolution is admissible with strict inequality in \eqref{eq:weakadm} for every $t\in(0,T)$. 
\end{definition}

\begin{theorem}\label{thm:rayleigh_taylor_subsolutions} 
 Let $n=2$, $\Omega=(0,1)\times (-c_-(T),c_+(T))$, where \[
c_-(t)=\frac{1}{2}\left(1-\sqrt{\frac{\rho_-}{\rho_+}}\right)gt^2,\quad c_+(t)=\frac{1}{2}\left(\sqrt{\frac{\rho_+}{\rho_-}}-1\right)gt^2.
\]
If $\rho_+>\left(\frac{4+2\sqrt{10}}{3}\right)^2\rho_-$, then
 there exists a RT-subsolution $z$ which only depends on $\frac{x_2}{t^2}$, and at time $t>0$ the mixing zone $\mathscr{U}(t):=\set{x\in\Omega:(x,t)\in\mathscr{U}}$ associated with $z$ is 
$(0,1)\times(-c_-(t),c_+(t))$.
\end{theorem}
An explicit description of the subsolutions and further discussion can be found after the proof of Theorem \ref{thm:rayleigh_taylor_subsolutions} in Section \ref{sec:subsolutions}.
Observe that by combining Theorem \ref{thm:main2} and Theorem \ref{thm:rayleigh_taylor_subsolutions} we arrive at the statement of Theorem \ref{thm:teaser_theorem}.

\section{Reformulation as a differential inclusion}\label{sec:reformulation_as_differential_inclusion}

The proof of Theorem \ref{thm:main2} will rely on a version of the Tartar framework for differential inclusions (cf e.g. \cite{Crippa, DeL-Sz-Adm, Tartar}), where instead of looking for weak solutions of a nonlinear problem, one looks for weak solutions of a first order linear PDE, satisfying a nonlinear algebraic constraint almost everywhere.

In order to reformulate \eqref{eq:main} into such a framework, we first observe that one can get rid of the gravity in the momentum equation  by considering the system in an accelerated domain. As mentioned earlier, this transformation corresponds to actual Rayleigh-Taylor experiments \cite{exp_dens,Read,exp2} where the instability is created by considering the stable configuration (light fluid above heavy fluid) and accelerating the surrounding container downwards.

To make this precise, let $\Omega\subset \R^n$ be a bounded domain, $T>0$ and set
\[
\mathscr{D}=\set{(y,t)\in\R^n\times(0,T):y-\frac{1}{2}gt^2e_n\in\Omega},
\]
such that for $t\in(0,T)$ the slice is given by 
\[
\mathscr{D}(t):=\set{y\in\R^n:(y,t)\in\mathscr{D}}=\Omega+\frac{1}{2}gt^2e_n.
\]
Let $(\mu,w,q)$ be a weak solution of
\begin{align}\label{eq:main0g}
\begin{split}
\partial_t(\mu w)+\text{div }(\mu w\otimes w) + \nabla q = 0,\\
\text{div } w=0,\\
\partial_t \mu + \text{div }(\mu w)=0,
\end{split}
\end{align}
on $\mathscr{D}$ for some suitable initial data satisfying \eqref{eq:id_comp} and with boundary condition
\[
(w(y,t)-gte_n)\cdot\nu_{\mathscr{D}(t)}(y)=0
\]
for $y\in\partial\mathscr{D}(t)$.
The notion of weak solution to \eqref{eq:main0g} is understood as in Definition \ref{def:weaksols}, except that now $g=0$, in the momentum and continuity equation $\Omega\times(0,T)$, $\Omega\times[0,T)$ is replaced by $\mathscr{D}$, $\mathscr{D}\cup (\Omega\times\{0\})$ resp., and the weak formulation of $\divv w=0$ including the boundary condition becomes
\begin{equation}\label{eq:divergence_free_w_with_boundary_data}
\int_{\mathscr{D}}w\cdot\nabla\Psi\:d(y,t)-\int_0^T\int_{\partial\mathscr{D}(t)}\Psi(y,t)gte_n\cdot\nu_{\mathscr{D}(t)}(y)\:dS(y)\:dt=0\text{ for all }\Psi\in\cC^\infty(\overline{\mathscr{D}}).
\end{equation}

Then if we define $y:=x+\frac{1}{2}gt^2e_n$ and set
\begin{align}\label{eq:transf}
\begin{split}
\rho(x,t)&=\mu\left(y,t\right),\\
v(x,t)&=w\left(y,t\right)-g t e_n,\\
p(x,t)&=q\left(y,t\right),
\end{split}
\end{align}
%\begin{align}\label{eq:transf}
%\begin{split}
%\rho(x,t)&=\mu\left(x+\frac{1}{2}gt^2e_n,t\right),\\
%v(x,t)&=w\left(x+\frac{1}{2}gt^2e_n,t\right)-g t e_n,\\
%p(x,t)&=q\left(x+\frac{1}{2}gt^2e_n,t\right).
%\end{split}
%\end{align}
it is straightforward to check that $(\rho,v)$ is a weak solution of \eqref{eq:main} on $\Omega\times (0,T)$ with the same initial data $(\rho_0,v_0)=(\mu_0,w_0)$. Observe also that the transformation \eqref{eq:transf} gives a bijective correspondence between solutions of \eqref{eq:main} and \eqref{eq:main0g}.

Furthermore, the formal energy associated with \eqref{eq:main0g} is given by the term $\frac{1}{2}\mu(y,t)\abs{w(y,t)}^2$.
Let us write 
\[
\frac{1}{2}\mu(y,t)\abs{w(y,t)}^2=\frac{n}{2}e(y-1/2gt^2e_n,t)
\]
for a function $e:\Omega\times[0,T)\rightarrow\R_+$. Then the total energy $E(x,t)$ associated with the original system \eqref{eq:main} is precisely given by \eqref{eq:relation_between_E_and_e}.

%Furthermore, if we denote the total energy $\frac{1}{2}\rho(x,t)|v(x,t)|^2+\rho(x,t) g x_n$ associated with \eqref{eq:main} by $E(x,t)$, we have that the total energy, which now consists only of the kinetic energy, associated with \eqref{eq:main0g} is given by
%\begin{align*}
%\frac{1}{2}\mu(x,t)|w(x,t)|^2=\frac{n}{2}(e(x,t)+\mu(x,t) h_0(x)+ \mu(x,t) h_1(t)\cdot w(x,t)),
%\end{align*}
%where
%\begin{align}\label{eq:energytrans}
%e(x,t)=\frac{2}{n}E\left(x-\frac{1}{2}gt^2e_n,t\right),\quad h_0(x)=-\frac{2}{n}gx_n,\quad h_1(t)=\frac{2}{n}gte_n.
%\end{align}

We can now reformulate \eqref{eq:main0g} as a differential inclusion by considering on $\mathscr{D}$ 
the system
\begin{align}\label{eq:linear_system}
\begin{split}
\partial_t m+\divv \sigma +\nabla q=0,\\
\divv w=0,\\
\partial_t\mu+\divv m=0,
\end{split}
\end{align}
where $z:=(\mu,w,m,\sigma,q)$ takes values in $Z=\R\times\R^n\times\R^n\times \mathcal S_0^{n\times n}\times \R$, together with the set of pointwise constraints
\begin{align}\label{eq:nonlinear_constraints_Kxt}
K_{(y,t)}=\left\{z\in Z: \mu\in\{\mu_-,\mu_+\},~m=\mu w,\ \mu w\otimes w-\sigma=e\left(y-\frac{1}{2}gt^2e_n,t\right)\id\right\},
\end{align}
where in analogy to the homogeneous Euler equations $e:\Omega\times(0,T)\rightarrow\R_+$ is given and for the sake of consistency we have denoted $\mu_\pm:=\rho_\pm$. 
We will understand weak solutions of \eqref{eq:linear_system} in the following sense.
\begin{definition}\label{def:translinweak}
We say that $z:\mathscr{D}\to Z$ is a weak solution of \eqref{eq:linear_system}
with initial data $\pi(z_0)\in L^2(\Omega;\pi(Z))$
iff $\pi(z)\in L^2(\mathscr{D};\pi(Z))$, $q$ is a distribution, $w$ satisfies \eqref{eq:divergence_free_w_with_boundary_data} and one has
\begin{align*}
\int_{\mathscr{D}} \left[m \cdot \partial_t \Phi + \langle \sigma,\nabla\Phi\rangle  \right] \ dx \ dt +\int_{\Omega}\mu_0(x)w_0 (x)\cdot \Phi(x,0)\ dx=0,\\
\int_{\mathscr{D}} \left[\mu \partial_t \Psi + m\cdot\nabla\Psi \right] \ dx \ dt +\int_{\Omega} \mu_0 (x) \Psi(x,0)\ dx=0,
\end{align*}
for any $\Phi\in C^\infty_c(\mathscr{D}\cup(\Omega\times\{0\});\mathbb{R}^n)$, $\divv \Phi =0$, $\Psi\in C^\infty_c(\mathscr{D}\cup(\Omega\times\{0\}))$.
\end{definition}

This way we have arrived at a reformulation of equation \eqref{eq:main} as a differential inclusion. The process is summarized in the following statement.
\begin{lemma}\label{lem:equivalence_between_pde_and_differential_inclusion}
Let $(\rho_0,v_0)\in L^\infty(\Omega)\times L^2(\Omega;\R^n)$ be initial data satisfying \eqref{eq:id_comp}, $e\in L^1(\Omega\times(0,T);\R_+)$ be a prescribed function. %and define $e(x,t)$, $h_0(x)$, $h_1(t)$ as in \eqref{eq:energytrans}. 
If $z=(\mu,w,m,\sigma,q)$ is a weak solution of \eqref{eq:linear_system} in the sense of Definition \ref{def:translinweak} with initial data $\mu(\cdot,0)=\rho_0$, $w(\cdot,0)=v_0$ and if $z(y,t)\in K_{(y,t)}$ for a.e. $(y,t)\in\mathscr{D}$, then the pair $(\rho,v)$ defined by \eqref{eq:transf} is a weak solution of \eqref{eq:main} on $\Omega\times(0,T)$ with initial data $(\rho_0,v_0)$. Moreover, the (possibly distributional) pressure is given by 
\[
p(x,t):=q(y,t)-\frac{1}{n}\mu(y,t)\abs{w(y,t)}^2,\quad y=x+\frac{1}{2}gt^2e_n,
\]
 and the associated energy $E$ by
\eqref{eq:relation_between_E_and_e}.
\end{lemma}

\section{The ingredients of the Tartar framework}\label{sec:conditions}

The general strategy of the Tartar framework relies on the following steps:
\begin{itemize}
\item finding a wave cone $\Lambda\subset Z$ such that for any $\bar{z}\in\Lambda$, one can construct a localized plane wave associated with \eqref{eq:linear_system} oscillating in the direction of $\bar{z}$;
\item calculating the $\Lambda$-convex hull of $K_{(x,t)}$ (denoted by $K_{(x,t)}^\Lambda$) and proving that one can perturb any element in its interior along sufficiently long $\Lambda$-segments, provided that one is far enough from $K_{(x,t)}$;
\item deducing an appropriate set of subsolutions using $K_{(x,t)}^\Lambda$ and proving that it is a bounded, nonempty subset of $L^2(\mathscr{D})$.
\end{itemize}

In the following subsections we execute each of the above steps in the case of the differential inclusion \eqref{eq:linear_system}, \eqref{eq:nonlinear_constraints_Kxt}. Then we can conclude the proof of Theorem \ref{thm:main2} in Section \ref{sec:conclusion} by using the Baire category method (see \cite{Cordoba,DeL-Sz-Annals,DeL-Sz-Adm,Kircheim,Tartar}).

\subsection{Localized plane waves}
We begin with the construction of plane wave-like solutions to \eqref{eq:linear_system} which are localized in space-time. 
We consider the following wave cone associated with \eqref{eq:linear_system}
\begin{align*}
\Lambda=\set{\bar{z}\in Z:\ker \begin{pmatrix}
\bar{\sigma} +\bar{q}\id & \bar{m}\\
\bar{m}^T & \bar{\mu}\\
\bar{w}^T & 0
\end{pmatrix} \neq \{0\},\quad (\bar{\mu},\bar{m})\neq0}.
\end{align*}
It has the property that for $\bar{z}\in\Lambda$ there exists $\eta\in\R^{n+1}\setminus\{0\}$ such that every $z(x,t)=\bar{z}h((x,t)\cdot\eta)$, $h\in\cC^1(\R)$ is a solution of \eqref{eq:linear_system}. In Lemma \ref{lem:locpw} below we localize these solutions by constructing suitable potentials. Note that the condition $(\bar{\mu},\bar{m})\neq0$ serves to eliminate the degenerate case when the first $n$ components of $\eta$ vanish, i.e. when one is only allowed to oscillate in time.

Recall the projection operator $\pi$ defined in \eqref{eq:projection}.

\begin{lemma}\label{lem:locpw}
There exists $C>0$ such that for any $\bar{z}\in\Lambda$, there exists a sequence 
$$z_N\in C_c^\infty(B_1(0);Z)$$
solving \eqref{eq:linear_system} and satisfying 
\begin{itemize}
\item[(i)] $d(z_N,[-\bar{z},\bar{z}])\to 0$ uniformly,
\item[(ii)] $z_N\rightharpoonup 0$ in $L^2$,
\item[(iii)] $\int\int |\pi(z_N)|^2\, dx \, dt\geq C|\pi(\bar{z})|^2.$
\end{itemize}
\end{lemma}
\begin{proof}
%\todo{only $2D$!! higher can be done but is harder, see Annals paper}

We will only present the proof in the two-dimensional case, higher dimensions can be handled analogously to \cite{DeL-Sz-Annals}.

We start by observing that for any smooth functions $\psi:\mathbb{R}^{2+1}\to\mathbb{R}$, $\phi:\mathbb{R}^{2+1}\to \mathcal S^{2\times 2}$, setting $D(\phi,\psi)=(\mu,w,m,\sigma,q)$ with
\begin{align*}
\mu=\divv\divv\phi,\quad
w=\nabla^{\perp} \psi,\quad
m=-\partial_t \divv \phi,\quad
q=\frac{1}{2}\tr \partial_{tt} \phi,\quad
\sigma=\partial_{tt} \phi-q\id,
\end{align*}
implies that $D(\phi,\psi)$ solves \eqref{eq:linear_system}.

Let $S:\mathbb{R}\to\mathbb{R}$ be a smooth function, $N\geq 1$ and $\bar{z}\in\Lambda$ with $(\bar{\mu},\bar{m})\neq0$. 
It follows that
there exists 
\begin{align}\label{eq:xc}
0\neq(\xi,c)\in\ker \begin{pmatrix}
\bar{\sigma} +\bar{q}\id & \bar{m}\\
\bar{m}^T & \bar{\mu}\\
\bar{w}^T & 0
\end{pmatrix}.
\end{align}
We then treat two cases.

\textbf{Case 1: $c\neq0$}

Note that in this case we also have $\xi\neq0$, since $\xi=0$ would imply $(\bar{\mu},\bar{m})=0$.

We then set
\begin{align*}
\phi_N(x,t)&=\frac{1}{c^2}(\bar{\sigma}+\bar{q}\id)\frac{1}{N^2}S(N(\xi,c)\cdot(x,t)),\\
\psi_N(x,t)&=|\bar{w}|\frac{\text{sgn}(\xi^\perp\cdot\bar{w})}{|\xi|}\frac{1}{N}S'(N(\xi,c)\cdot(x,t)),
\end{align*}
and we claim that
\begin{align}\label{eq:c1pot}
D(\phi_N,\psi_N)=\bar{z}S''(N(\xi,c)\cdot(x,t)).
\end{align}
Indeed, using \eqref{eq:xc}, one has
\begin{align*}
\divv\divv\phi_N&=\frac{1}{c^2}\xi^T(\bar{\sigma}+\bar{q}\id)\xi S''(N(\xi,c)\cdot(x,t))\\&=\frac{1}{c^2}\xi^T(-c\bar{m})S''(N(\xi,c)\cdot(x,t))=\bar{\mu}S''(N(\xi,c)\cdot(x,t)),\\
\partial_t\divv \phi_N&=\frac{1}{c}(\bar{\sigma}+\bar{q}\id)\xi S''(N(\xi,c)\cdot(x,t))=-\bar{m} S''(N(\xi,c)\cdot(x,t)),\\
\partial_{tt}\phi_N&=c^2\frac{1}{c^2}(\bar{\sigma}+\bar{q}\id) S''(N(\xi,c)\cdot(x,t))=(\bar{\sigma}+\bar{q}\id) S''(N(\xi,c)\cdot(x,t)),\\
\nabla^{\perp}\psi_N&=\xi^{\perp}|\bar{w}|\frac{\text{sgn}(\xi^\perp\cdot\bar{w})}{|\xi|}S''(N(\xi,c)\cdot(x,t)) =\bar{w} S''(N(\xi,c)\cdot(x,t)).
\end{align*}

From here on, the localization is done in the standard fashion (e.g. as in \cite{Cordoba,DeL-Sz-Annals}). We fix $S(\cdot)=-\cos(\cdot)$ and, for $\varepsilon>0$, consider $\chi_\varepsilon\in C_c^\infty(B_1(0))$ satisfying $|\chi_\varepsilon|\leq 1$ on $B_{1}(0)$, $\chi_\varepsilon=1$ on $B_{1-\varepsilon}(0)$. It is then straightforward to check that 
$z_N=D(\chi_\varepsilon(\phi_N,\psi_N))$ satisfies the conclusions of the lemma.

\textbf{Case 2: $c=0$}

In this case we are not allowed to oscillate in time.
However, we have $\xi\neq0$, so we may also assume without loss of generality that $|\xi|=1$. On the other hand, \eqref{eq:xc} implies that there exist constants $k_1,k_2,k_3\in\mathbb{R}$ such that 
\begin{align}\label{eq:perps}
\bar{w}=k_1\xi^\perp,\quad
\bar{m}=k_2\xi^\perp,\quad
\bar{\sigma}+\bar{q}\id=k_3\xi^\perp\otimes\xi^\perp
.\end{align}
We set
\begin{align*}
\phi_{N}(x,t)=\bar{\mu}\id\frac{1}{N^2}S(N\xi\cdot x),\quad
\psi_{N}(x,t)=|\bar{w}|\frac{\text{sgn}(\xi^\perp\cdot\bar{w})}{|\xi|}\frac{1}{N}S'(N\xi\cdot x),
\end{align*}
from where with similar calculations as in Case 1, we obtain that
\begin{align}\label{eq:c21pot}
D(\phi_{N},\psi_{N})=(\bar{\mu},\bar{w},0,0,0)S''(N\xi\cdot x).
\end{align}

To handle the remaining terms $(\bar{m},\bar{\sigma},\bar{q})$, we introduce a different type of potential, as done for the homogeneous Euler equations, for instance in \cite{DeL-Sz-Annals}, Remark 2.

It can be checked through direct calculation that for any smooth function $\omega:\mathbb{R}^{2+1}\to\mathbb{R}^{2+1}$, defining $W=\curl_{(x,t)}\omega$ and $\tilde{D}(\omega)=(0,0,m,\sigma,q)$ by
\begin{align*}
m=-\frac{1}{2}\nabla^\perp W_3,\quad
\sigma+q\id=\begin{pmatrix}
\partial_2 W_1 & \frac{1}{2}(\partial_2W_2-\partial_1W_1)\\
\frac{1}{2}(\partial_2W_2-\partial_1W_1) & -\partial_1W_2
\end{pmatrix}
\end{align*}
implies that $\tilde{D}(\omega)$ solves \eqref{eq:linear_system}.

Now, if we consider $\omega$ of the form
$$\omega_N(x)=(a,b,a)\frac{1}{N^2} S(N \xi\cdot x),$$
for some constants $a,b\in\mathbb{R}$, with $S$ as before, we obtain that
\begin{gather*}
\begin{pmatrix}
\partial_2 W_1 & \frac{1}{2}(\partial_2W_2-\partial_1W_1)\\
\frac{1}{2}(\partial_2W_2-\partial_1W_1) & -\partial_1W_2
\end{pmatrix}=a\xi^\perp\otimes\xi^\perp S''(N \xi\cdot x),\\
\nabla^\perp W_3=(\xi_1b-\xi_2a)\xi^\perp S''(N \xi\cdot x).
\end{gather*}

If $\xi_1\neq 0$, it follows from \eqref{eq:perps} that setting $a=k_3$, $b=\frac{-2k_2+k_3\xi_2}{\xi_1}$ gives us
\begin{align*}
\tilde{D}(\omega_{N})=(0,0,\bar{m},\bar{\sigma},\bar{q})S''(N\xi\cdot x).
\end{align*}
from where, using \eqref{eq:c21pot}, we get
\begin{align*}
D(\phi_{N},\psi_{N})+\tilde{D}(\omega_{N})=\bar{z} S''(N\xi\cdot x).
\end{align*}
The localization is then done as in Case 1, by considering $z_N=D(\chi_\varepsilon(\phi_{N},\psi_{N}))+\tilde{D}(\chi_\varepsilon\omega_{N}).$

If $\xi_1=0$, then choosing $a=k_3$ gives us that 
\begin{align*}
\tilde{D}(\omega_{N})=\left(0,0,\frac{k_3}{2}\xi_2\xi^\perp,\bar{\sigma},\bar{q}\right)S''(N\xi\cdot x).
\end{align*}
However, it is easy to see that for any smooth function $\theta:\mathbb{R}^{2+1}\to\mathbb{R}$, $\hat{D}(\theta)=(0,0,\nabla^\perp\theta,0,0)$ also solves \eqref{eq:linear_system}. 
Therefore, we may consider the potential given by
$$\theta_N(x)=\left(k_2-\xi_2\frac{k_3}{2}\right)\frac{1}{N}S'(N\xi\cdot x),$$
we obtain that
$$\nabla^\perp\theta_N(x)=\left(k_2-\xi_2\frac{k_3}{2}\right)\xi^\perp S''(N\xi\cdot x),$$
and  using \eqref{eq:perps}, we get that
\begin{align*}
D(\phi_{N},\psi_{N})+\tilde{D}(\omega_{N})+\hat{D}(\theta_N)=\bar{z} S''(N\xi\cdot x).
\end{align*}
One may then localize this potential by the usual means in order to conclude the proof of the lemma.
\end{proof}

\subsection{\texorpdfstring{The $\Lambda$-convex hull}{The convex hull}}\label{sec:convex_hull}
%\todo{Some of the definitions here will in the end be contained in previous sections.}
%The verification of (H2) requires a good enough understanding of the convex hull $K_{(x,t)}^{co}$ or more generally t
We now turn to the set of pointwise constraints $K_{(x,t)}$, $(x,t)\in\mathscr{D}$ defined in \eqref{eq:nonlinear_constraints_Kxt}.
The $\Lambda$-convex hull $K_{(x,t)}^\Lambda$ is defined by saying that $z\in K_{(x,t)}^\Lambda$ iff for all $\Lambda$-convex functions $f:Z\rightarrow\R$ there holds $f(z)\leq \sup_{z'\in K_{(x,t)}}f(z')$, see \cite{Kircheim} for more details. In our case it turns out that the $\Lambda$-convex hull is nothing else but the usual convex hull, see Proposition \ref{prop:lambda_convex_hull} below.

For the computation of the hull we drop the $(x,t)$ dependence of the sets $K_{(x,t)}$ and consider a general set of pointwise constraints given by
\begin{align}\label{eq:constraints}
K=\set{z\in Z: \mu\in\{\mu_-,\mu_+\},~m=\mu w,~\mu w\otimes w-\sigma =e\id},
\end{align}
where $0<\mu_-<\mu_+$, $e\in\R_+$ are given constants. %and $h_1\in\R^n$ is a given vector satisfying
%\begin{equation}\label{eq:condition_on_constants_for_K_nonempty}
%e+\mu_+h_0+\mu_+\frac{n}{4}\abs{h_1}^2\geq0,\quad e+\mu_-h_0+\mu_-\frac{n}{4}\abs{h_1}^2\geq0.
%\end{equation}
%As one can see by taking the trace of $\mu w\otimes w-\sigma$, $z\in K$ these conditions are necessary for $K_\pm:=K\cap \set{z\in Z:\mu=\mu_\pm}$ to be non-empty. In the $(x,t)$ dependent case, when $e,h_0,h_1$ are given by \eqref{eq:energytrans} observe that in view of the transformations \eqref{eq:energytrans}, if $E$ satisfies \eqref{eq:energycond},
%then \eqref{eq:condition_on_constants_for_K_nonempty} also holds.

%\todo{To be added later: Remark about $h_0$, $h_1$ and $e$. Backtransformation of $\overline{U}$ to $(\rho,v)$ coordinates.}

Define $Z_0:=\set{z\in Z:\mu\in(\mu_-,\mu_+)}$ and $T_+,T_-,Q:Z_0\rightarrow \R$, $M:Z_0\rightarrow \cS^{n\times n}$,
\begin{gather*}
M(z)=\frac{\mu\mu_-\mu_+w\otimes w-\mu_-\mu_+(m\otimes w+w\otimes m)+(\mu_++\mu_--\mu)m\otimes m}{(\mu_+-\mu)(\mu-\mu_-)}-\sigma,\\
Q(z)=\lamax(M(z)),\quad T_{\pm}(z)=\frac{\mu_\pm}{n}\frac{\abs{m-\mu_\mp w}^2}{(\mu-\mu_\mp)^2},
\end{gather*}
as well as the open set 
\begin{equation}\label{eq:definition_of_U}
U=\set{z\in Z:\mu\in(\mu_-,\mu_+),~T_+(z)<e,~T_-(z)<e,~Q(z)<e}.
\end{equation}

\begin{proposition}\label{prop:lambda_convex_hull}
The $\Lambda$-convex hull of $K$ coincides with the convex hull of $K$ and is given by $\overline{U}$, i.e., $K^\Lambda=K^{co}=\overline{U}$.
\end{proposition}

Lemma \ref{lem:U_convex_closure_of_U} below shows that the closure of $U$ can be written as 
\[
\overline{U}=K_-'\cup \overline{U}_0\cup K_+',
\]
where
\begin{align*}
\overline{U}_0&=\set{z\in Z:\mu\in(\mu_-,\mu_+),~T_+(z)\leq e,~T_-(z)\leq e,~Q(z)\leq e},\\
K_\pm'&=\{z\in Z:\mu=\mu_\pm,~m=\mu_\pm w,~\lamax(\mu_\pm w\otimes w-\sigma)\leq e\}.
\end{align*}
Moreover, Lemma \ref{lem:boundedness_of_U_extreme_points} actually shows that $K_+'$, $K_-'$ resp., is nothing but the $\Lambda$-convex hull of $K_+:=K\cap\{\mu=\mu_+\}$, $K_-:=K\cap\{\mu=\mu_-\}$ resp.. 

Furthermore, notice that if one lets $\mu_+-\mu_-\to 0$, one recovers from $\overline{U}$ exactly the convex hull of the constraints for the homogeneous Euler equations, cf. \cite{DeL-Sz-Adm}.

The proof of Proposition \ref{prop:lambda_convex_hull} relies on Lemma \ref{lem:U_convex_closure_of_U} and \ref{lem:boundedness_of_U_extreme_points}.

\begin{lemma}\label{lem:Q_convex} 
The function $Q$ is convex.
\end{lemma}
\begin{proof}
We write
\[
Q(z)=\sup_{\xi\in S^{n-1}}\xi^TM(z)\xi=\sup_{\xi\in S^{n-1}}\left(g_\xi(z)-\xi^T\sigma \xi\right),
\]
where for every fixed $\xi\in S^{n-1}$ the function $g_\xi:Z_0\rightarrow \R$ is given by 
\begin{align*}
g_\xi(z)&=\xi^TM(z)\xi+\xi^T\sigma \xi\\
&=\frac{\mu\mu_-\mu_+(w\cdot\xi)^2-2\mu_-\mu_+(m\cdot\xi)(w\cdot\xi)+(\mu_++\mu_--\mu)(m\cdot\xi)^2}{(\mu_+-\mu)(\mu-\mu_-)}.
\end{align*}
We will show that every $g_\xi$ is convex. As a consequence $Q$ is convex as a supremum of convex functions. In order to do this let us complement $\xi\in S^{n-1}$ to a orthonormal basis $(\xi,v_2,\ldots,v_n)$ of $\R^n$. Expressing $w$ and $m$ with respect to this basis one sees that it is enough to show that the function $g:(\mu_-,\mu_+)\times\R^2\rightarrow \R$,
\begin{align*}
g(\mu,x)=\frac{\mu\mu_-\mu_+x_1^2-2\mu_-\mu_+x_1x_2+(\mu_++\mu_--\mu)x_2^2}{(\mu_+-\mu)(\mu-\mu_-)}
\end{align*}
is convex. We write $g(\mu,x)=x^TA(\mu)x$ with
\[
A(\mu):=\frac{1}{(\mu_+-\mu)(\mu-\mu_-)}\begin{pmatrix}
\mu\mu_-\mu_+ & -\mu_-\mu_+ \\
-\mu_-\mu_+ & \mu_++\mu_--\mu
\end{pmatrix}.
\] 
Let us fix $(\mu,x)\in(\mu_-,\mu_+)\times\R^2$ and observe that $A(\mu)$ is positive definite because $\mu\mu_-\mu_+>0$ and 
\[
\det [(\mu_+-\mu)(\mu-\mu_-)A(\mu)]=\mu_-\mu_+(\mu_+-\mu)(\mu-\mu_-)>0.
\]
Thus the restricted function $g(\mu,\cdot)$ is convex, or equivalently $D^2g(\mu,x)[0,y]^2\geq 0$ for all $y\in\R^2$. It therefore remains to show that $D^2g(\mu,x)[1,y]^2\geq 0$ for all $y\in\R^2$. By the positive definiteness of $A(\mu)$ we obtain
\begin{align*}
D^2g(\mu,x)[1,y]^2&=x^TA''(\mu)x+4y^TA'(\mu)x +2y^TA(\mu)y\\
&=2\left(y+A(\mu)^{-1}A'(\mu)x\right)^TA(\mu)\left(y+A(\mu)^{-1}A'(\mu)x\right)\\
&\phantom{=asd}+x^TA''(\mu)x -2x^TA'(\mu)A(\mu)^{-1}A'(\mu)x\\
&\geq x^T\left(A''(\mu) -2A'(\mu)A(\mu)^{-1}A'(\mu)\right)x.
\end{align*}
Now we claim that in fact $A''(\mu)=2A'(\mu)A(\mu)^{-1}A'(\mu)$, which finishes the proof.
Indeed, differentiation of the identity 
\[
(\mu_+-\mu)(\mu-\mu_-)A(\mu)=\begin{pmatrix}
\mu\mu_-\mu_+ & -\mu_-\mu_+ \\
-\mu_-\mu_+ & \mu_++\mu_--\mu
\end{pmatrix}
\]
shows that 
\begin{align}
(\mu_+-\mu)(\mu-\mu_-)A'(\mu)&=(2\mu-\mu_--\mu_+)A(\mu)+C,\label{eq:derivativeA}\\
\begin{split}
(\mu_+-\mu)^2(\mu-\mu_-)^2A''(\mu)&=2((\mu_+-\mu)(\mu-\mu_-)+(2\mu-\mu_--\mu_+)^2)A(\mu)\\
&\phantom{=jasnkfagf}+2(2\mu-\mu_--\mu_+)C,
\end{split}
\end{align}
where
\[
C:=\begin{pmatrix}
\mu_-\mu_+ & 0\\
0 & -1
\end{pmatrix}.
\]
Moreover, a straightforward computation yields
\begin{align}\label{eq:CAinversC}
CA(\mu)^{-1}C=(\mu_+-\mu)(\mu-\mu_-)A(\mu)+(\mu_-+\mu_+-2\mu)C.
\end{align}
Now \eqref{eq:derivativeA}--\eqref{eq:CAinversC} imply the identity $A''(\mu)=2A'(\mu)A^{-1}(\mu)A'(\mu)$.
\end{proof}
\begin{lemma}\label{lem:U_convex_closure_of_U} The set $U$ is convex and $\overline{U}=K_-'\cup\overline{U}_0\cup K_+'$.
In particular $K\subset \overline{U}$.
\end{lemma}
\begin{proof}
For $\mu\in(\mu_-,\mu_+)$ the two conditions $T_+(z)<e$, $T_-(z)<e$ can be rewritten as
\begin{align}\begin{split}\label{eq:bounds_given_by_Tplusminus}
\abs{m-\mu_- w}&<c_+ (\mu-\mu_-),\\
\abs{m-\mu_+ w}&<c_- (\mu_+-\mu),
\end{split}
\end{align}
where 
$
c_\pm=\left(\frac{n e}{\mu_\pm}\right)^{1/2}.
$ %Note that the constants $c_\pm$ are non-negative by \eqref{eq:condition_on_constants_for_K_nonempty}.
Using the basic triangle inequality one can check that the two conditions in \eqref{eq:bounds_given_by_Tplusminus} define a convex set. By Lemma \ref{lem:Q_convex} we already know that $Q$ is a convex function. Hence we have shown that $U$ is convex.

Now we turn to the characterization of $\overline{U}$. Clearly $\overline{U}_0\subset\overline{U}$. Let us show that $K_+'\subset\overline{U}$. The inclusion $K_-'\subset\overline{U}$ can be obtained in the same way. Let $z_*\in K_+'$. Take any $z'\in K$ with $\mu'=\mu_-$ and some sequence $(\mu_j)_{j\in\N}\subset (\mu_-,\mu_+)$ converging to $\mu_+$. Define
\begin{align*}
z_j=\frac{\mu_+-\mu_j}{\mu_+-\mu_-}z'+\frac{\mu_j-\mu_-}{\mu_+-\mu_-}z_*.
\end{align*}
Clearly $z_j\rightarrow z_*$ as $j\rightarrow \infty$. Since $z_*\in K_+'$ and $z'\in K_-$ a short calculation shows 
\begin{align*}
T_+(z_j)=\frac{\mu_+}{n}\abs{w_*}^2
=\frac{1}{n}\tr(\mu_+w_*\otimes w_*-\sigma_*)
\leq \lamax(\mu_+w_*\otimes w_*-\sigma_*)
\leq e.
\end{align*}
Similarly we obtain $T_-(z_j)=e$ and 
\begin{align*}
M(z_j)&=\frac{\mu_+-\mu_j}{\mu_+-\mu_-}\big(\mu_-w'\otimes w'-\sigma'\big)
+\frac{\mu_j-\mu_-}{\mu_+-\mu_-}\big(\mu_+w_*\otimes w_*-\sigma_*\big)\\
&=\frac{\mu_+-\mu_j}{\mu_+-\mu_-}e\id+\frac{\mu_j-\mu_-}{\mu_+-\mu_-}\big(\mu_+w_*\otimes w_*-\sigma_*\big).
\end{align*}
We conclude $Q(z_j)=\lamax(M(z_j))\leq e$.
Hence every $z_j$ and therefore also the limit $z_*$ is contained in $\overline{U}$. So far we know $K_-'\cup \overline{U}_0\cup K_+'\subset\overline{U}$.

For the other inclusion consider $(z_j)_{j\in\N}\subset U$, $z_j\rightarrow z_*$. The interesting case of course is $\mu_*\notin(\mu_-,\mu_+)$, say $\mu_*=\mu_+$. By \eqref{eq:bounds_given_by_Tplusminus} we directly see that $m_*=\mu_+w_*$. Moreover, rewriting
\begin{align*}
M(z)=\mu_-\frac{m-\mu w}{\mu-\mu_-}\otimes\frac{m-\mu_+w}{\mu_+-\mu}+\frac{m-\mu_-w}{\mu-\mu_-}\otimes m-\sigma,
\end{align*}
and a look at \eqref{eq:bounds_given_by_Tplusminus} yields
\[
\lim_{j\rightarrow \infty} M(z_j)= \mu_+w_*\otimes w_*-\sigma_*.
\]
Thus $\lamax(M(z_j))<e$, $j\in \N$ implies $z_*\in K_+'$. The case $\mu_*=\mu_-$ can again be treated by obvious adaptions. Consequently $\overline{U}= K_-'\cup \overline{U}_0\cup K_+'$.
\end{proof}
Next we introduce the most important $\Lambda$-directions.
\begin{definition}\label{def:directions}
Let $z\in Z_0$. We call $\tilde{z}(z)\in Z$ defined by 
\begin{gather*}
\tilde{\mu}=1,\quad
\tilde{w}(z)=\frac{m-\mu w}{(\mu_+-\mu)(\mu-\mu_-)},\quad
\tilde{m}(z)=w+(\mu_++\mu_--\mu)\tilde{w}(z),\\
\tilde{\sigma}(z)+\tilde{q}(z)\id =\tilde{m}(z)\otimes \tilde{m}(z)-\mu_+\mu_-\tilde{w}(z)\otimes\tilde{w}(z)
\end{gather*}
the Muskat direction associated with $z$. Here the definition of $\tilde{q}$ and $\tilde{\sigma}$ is understood as decomposition into trace and traceless part. Moreover, any vector of the form $\bar{z}=(0,\bar{w},\lambda\bar{w},\bar{\sigma},\bar{q})$, $\lambda\in\R$ is called an Euler direction provided it is contained in the wave cone $\Lambda$.
\end{definition}

Note that the Euler direction comes from the perturbations used in \cite{DeL-Sz-Annals} for the homogeneous incompressible Euler equations, while the Muskat direction is a generalization of the perturbations introduced in \cite{Sz-Muskat} for the Muskat problem  (hence the name), having the property of conserving the quantity $\frac{m-\mu w}{(\mu_+-\mu)(\mu-\mu_-)}$, as seen in the proof of the following Lemma.

\begin{lemma}\label{lem:properties_of_directions} There holds
\begin{enumerate}[(i)]
\item For any pair $(\bar{w},\bar{\sigma})\in \R^n\times  \mathcal S_0^{n\times n}$, $\bar{w}\neq0$, there exists $\bar{q}\in\R$, such that for $\lambda\in\R\setminus\{0\}$ the vector $\bar{z}=(0,\bar{w},\lambda\bar{w},\bar{\sigma},\bar{q})$ is an Euler direction.
\item The Muskat directions $\tilde{z}(z)$, $z\in Z_0$ are contained in $\Lambda$.
\item For $z\in Z_0$ define $z_t:=z+t\tilde{z}(z)$, $t\in(\mu_--\mu,\mu_+-\mu)$. Then $\tilde{z}(z_t)$, $T_\pm(z_t)$ and the traceless part $M(z_t)^\circ$ are all independent of $t$.
\item $T_+(z+t\bar{z})=T_+(z)$ for all $t\in\R$ and all Euler directions $\bar{z}$ with $\bar{m}=\mu_-\bar{w}$, as well as $T_-(z+t\bar{z})=T_-(z)$ for all $t\in\R$ and all Euler directions of the form $\bar{z}=(0,\bar{w},\mu_+\bar{w},\bar{\sigma},\bar{q})$.
\end{enumerate}
\end{lemma} 
\begin{proof}
(i) This basically has been shown in \cite{DeL-Sz-Adm}. We nonetheless present the short proof here as well. Let $(\bar{w},\bar{\sigma},\lambda)\in\R^n\times  \mathcal S_0^{n\times n}\times\R$, %If $\bar{w}=0$ or $\lambda=0$, then any $\bar{q}\in\R$ does the job. 
 $\bar{w}\neq 0$, $\lambda\neq 0$ and denote by $P_\perp:\R^n\rightarrow \R^n$ the orthogonal projection onto $\bar{w}^\perp$. Take $\bar{q}\in\R$, such that $-\bar{q}$ is an eigenvalue of the linear map $P_\perp\circ \bar{\sigma}:\bar{w}^\perp\rightarrow\bar{w}^\perp$, and let $\xi\in\bar{w}^\perp\setminus\{0\}$ denote a corresponding eigenvector. Furthermore, we choose $c\in\R$, such that $(\id-P_\perp)\bar{\sigma}\xi=-c\lambda\bar{w}$. Then one easily checks that 
\[
\begin{pmatrix}
\bar{\sigma}+\bar{q}\id & \lambda\bar{w}\\
\lambda\bar{w}^T & 0\\
\bar{w}^T & 0
\end{pmatrix}\begin{pmatrix}
\xi\\c
\end{pmatrix}=0.
\]

(ii) Let $z\in Z_0$, take any element $\xi\in \R^n\setminus\{0\}$ with $\tilde{w}(z)\cdot\xi=0$ and define $c:=-\tilde{m}(z)\cdot\xi$. Then 
\begin{align*}
\begin{pmatrix}
\tilde{\sigma}(z)+\tilde{q}(z)\id & \tilde{m}(z)\\
\tilde{m}(z)^T & 1\\
\tilde{w}(z)^T & 0
\end{pmatrix}\begin{pmatrix}
\xi \\
c
\end{pmatrix}=(\tilde{m}(z)\cdot \xi+c)\begin{pmatrix}
\tilde{m}(z)\\
1\\0
\end{pmatrix}=0.
\end{align*}

(iii) Let $z\in Z_0$, $t\in(\mu_--\mu,\mu_+-\mu)$, $z_t=z+t\tilde{z}(z)$. First of all observe that
\begin{align*}
(\mu_+-\mu-t)(\mu+t-\mu_-)\tilde{w}(z_t)&=m+t\tilde{m}(z)-(\mu+t)(w+t\tilde{w}(z))\\
&=m-\mu w +t(\mu_++\mu_--2\mu)\tilde{w}(z)-t^2\tilde{w}(z)\\
&=(\mu_+-\mu-t)(\mu+t-\mu_-)\tilde{w}(z).
\end{align*}
Hence $\tilde{w}(z_t)=\tilde{w}(z)$ and
\begin{align*}
\tilde{m}(z_t)&=w+t\tilde{w}(z)+(\mu_++\mu_--\mu-t)\tilde{w}(z_t)\\
&=w+(\mu_++\mu_--\mu)\tilde{w}(z)=\tilde{m}(z).
\end{align*}
The invariances $\tilde{\sigma}(z_t)=\tilde{\sigma}(z)$ and $\tilde{q}(z_t)=\tilde{q}(z)$ then follow by the definition of $\tilde{\sigma}$, $\tilde{q}$. Thus $\tilde{z}(z_t)=\tilde{z}(z)$.

Next $T_\pm(z_t)=T_\pm(z)$ follows immediately after rewriting 
\begin{align*}
T_+(z)=\frac{\mu_+}{n}\abs{w+(\mu_+-\mu)\tilde{w}(z)}^2,\quad
T_-(z)=\frac{\mu_-}{n}\abs{w+(\mu_--\mu)\tilde{w}(z)}^2.
\end{align*}
It remains to check that the traceless part of $M(z)$ is invariant along the line segment in Muskat direction. Plugging 
\begin{align*}
w&=\tilde{m}(z)-(\mu_++\mu_--\mu)\tilde{w}(z),\\
m&=\mu w +(\mu_+-\mu)(\mu-\mu_-)\tilde{w}(z)=\mu \tilde{m}(z)-\mu_-\mu_+\tilde{w}(z)
\end{align*}
into the definition of $M(z)$ leads us to
\begin{align*}
M(z)&=\mu\tilde{m}(z)\otimes \tilde{m}(z)-\mu_-\mu_+(\tilde{m}(z)\otimes\tilde{w}(z)+\tilde{w}(z)\otimes \tilde{m}(z))\\
&\phantom{=asd}+\mu_-\mu_+(\mu_++\mu_--\mu)\tilde{w}(z)\otimes\tilde{w}(z)-\sigma.
\end{align*}
Thus for the traceless part we get
\begin{align*}
M(z_t)^\circ&=M(z)^\circ+t\big(\tilde{m}(z)\otimes\tilde{m}(z)-\mu_-\mu_+\tilde{w}(z)\otimes \tilde{w}(z)\big)^\circ-t\tilde{\sigma}(z)=M(z)^\circ.
\end{align*}

(iv) obviously is true, because $m+t\bar{m}-\mu_{\pm}(w+t\bar{w})=m-\mu_\pm w$ for $\bar{m}=\mu_\pm \bar{w}$. 
\end{proof}
As a corollary, we obtain that any two points in $K$ can be connected with a $\Lambda$-direction, up to modifying the pressure, which implies that although the wave cone $\Lambda$ is not the whole space, it is still quite big (with respect to $K$).
\begin{corollary}\label{bigenoughcone}
For any $z_1,z_2\in K$, $z_1\neq z_2$, %there exists $\bar{q}\in\mathbb{R}$ such that 
one has
$z_2-z_1+(0,0,0,0,q_1-q_2)\in\Lambda.$
\end{corollary}
\begin{proof}
In the case $\mu_1\neq \mu_2$ we assume without loss of generality that $\mu_1=\mu_{-}$ and $\mu_2=\mu_{+}$. Set $\bar{z}=z_2-z_1+(0,0,0,0,q_1-q_2)$, such that $\bar{q}=0$.
Similarly to (ii) from Lemma \ref{lem:properties_of_directions} one can prove that $\bar{z}\in \Lambda$ if 
\begin{align}\label{coneref}
\bar{\mu}(\bar{\sigma}+\bar{q}\id)=\bar{m}\otimes\bar{m}+\gamma\bar{w}\otimes\bar{w},
\end{align}
for some $\gamma\in\mathbb{R}.$ 

Since $z_i\in K$, we have 
$$\sigma_i=\mu_i w_i\otimes w_i-e\id,$$
for $i=1,2.$ Therefore, we obtain that
$$\bar{\sigma}=(\bar{\sigma}+\bar{q}\id)=\mu_2 w_2\otimes w_2-\mu_1 w_1\otimes w_1.$$
Through a simple calculation one can then show that \eqref{coneref} holds for $\gamma=-\mu_{-}\mu_{+}$.

If $\mu_1=\mu_2$, recall that in the proof of Lemma \ref{lem:properties_of_directions} (i) a suitable pressure $\bar{q}$ has been choosen to be an eigenvalue of $-P_\perp\circ\bar{\sigma}:\bar{w}^\perp\rightarrow\bar{w}^\perp$. But $z_1,z_2\in K$ in fact implies that $P_\perp\circ \bar{\sigma}$ vanishes on all of $\bar{w}^\perp$ and we can conclude the statement.
\end{proof}
Recall the definition of $\pi:Z\rightarrow\R\times\R^n\times\R^n\times \cS_0^{n\times n}$ in \eqref{eq:projection}.
\begin{lemma}\label{lem:boundedness_of_U_extreme_points}
The projection $\overline{U/\sim}:=\overline{\pi(U)}$ is bounded in terms of $e$, $\mu_\pm$, $n$ and hence compact. Moreover, for every $z\in\overline{U}\setminus K$ there exists $\bar{z}\in \Lambda\setminus\{0\}$, such that $z\pm\bar{z}\in \overline{U}$.
\end{lemma}
\begin{proof}
We first prove that $U/\sim$ is bounded in terms of $e,\mu_-,\mu_+$ and the dimension $n$. Let $z\in U$. Obviously $\mu\in(\mu_-,\mu_+)$ is bounded. The inequalities \eqref{eq:bounds_given_by_Tplusminus} imply that there exists  a constant $c=c(e,\mu_-,\mu_+,n)>0$, such that 
\begin{align}\label{eq:bounds_by_Tplusminus_2}
\abs{m-\mu_-w}\leq c(\mu-\mu_-),\quad \abs{m-\mu_+w}\leq c(\mu_+-\mu).
\end{align}
Adapting the constant when necessary we obtain
\[
\abs{m}=\abs{\frac{\mu_+}{\mu_+-\mu_-}(m-\mu_-w)-\frac{\mu_-}{\mu_+-\mu_-}(m-\mu_+w)}\leq c,
\]
which then also implies $\abs{w}\leq c$. Next observe that the matrix $M(z)$ can be rewritten to
\begin{align*}
M(z)&=-\mu\frac{m-\mu_-w}{\mu-\mu_-}\otimes\frac{m-\mu_+w}{\mu-\mu_+}+\frac{m-\mu_-w}{\mu-\mu_-}\otimes m +m\otimes \frac{m-\mu_+w}{\mu-\mu_+}-\sigma.
\end{align*}
Hence $M(z)+\sigma$ is uniformly bounded by \eqref{eq:bounds_by_Tplusminus_2}. As a consequence we obtain $\abs{\tr M(z)}\leq c$. This bound on the trace together with $\lamax(M(z))=Q(z)<e$, due to the fact that $z\in U$, gives us a uniform bound on the whole spectrum of $M(z)$. Therefore $M(z)+\sigma$ and $M(z)$ are both uniformly bounded. Consequently $\abs{\sigma}\leq c$, and $\overline{U/\sim}$ is compact.

Next we show that any $z\in\overline{U}\setminus K$ can be perturbed along a $\Lambda$-segment without leaving $\overline{U}$. Recall that $\overline{U}=\overline{U}_0\cup K_+'\cup K_-'$ and $K\subset K_+'\cup K_-'$ by Lemma \ref{lem:U_convex_closure_of_U}. 

If $z\in K_+'\setminus K$, we can find similarly as in \cite{DeL-Sz-Adm} a suitable Euler direction $\bar{z}=(0,\bar{w},\mu_+\bar{w},\bar{\sigma},\bar{q})\in \Lambda$ such that $z+t\bar{z}\in K_+'$ for $\abs{t}$ small enough. Indeed, by a change of basis we can restrict ourselves to the case that $\mu_+w\otimes w-\sigma$ is diagonal. Denote the entries by $\lambda_1\geq \lambda_2 \geq \ldots \geq \lambda_n$, where $\lambda_1 \leq e$ and $\lambda_n<e$. Let $e_1,\ldots,e_n$ denote the canonical basis of $\R^n$. We take $\bar{w}=e_n$ and 
\[
\bar{\sigma}=\mu_+e_n\otimes w+\mu_+w\otimes e_n-\alpha e_n\otimes e_n,
\]
where $\alpha=2\mu_+w_n$ makes $\bar{\sigma}$ trace free. It follows
\begin{align*}
\mu_+(w+t\bar{w})\otimes &(w+t\bar{w})-(\sigma+t\bar{\sigma})\\
&=\sum_{j=1}^n\lambda_je_j\otimes e_j+t(\mu_+e_n\otimes w+\mu_+w\otimes e_n-\bar{\sigma})+t^2\mu_+e_n\otimes e_n\\
&=\sum_{j=1}^{n-1}\lambda_je_j\otimes e_j + (\lambda_n+\alpha t+\mu_+t^2)e_n\otimes e_n.
\end{align*}
Clearly, $\lambda_j\leq e$, $j=1,\ldots,n-1$ and
$\lambda_n+\alpha t+\mu_+t^2\leq e$ for all $\abs{t}$ small enough, since the inequaltiy holds strict for $t=0$. 

The same reasoning applies also to the case $z\in K_-'\setminus K$.

Now let $z\in \overline{U}_0$. If $Q(z)<e$ or if $T_+(z)=T_-(z)$ we take the Muskat direction $\bar{z}=\tilde{z}(z)$. Because then
$
T_\pm(z+t\tilde{z}(z))=T_\pm(z)\leq e
$, $t\in(\mu_--\mu,\mu_+-\mu)$ by Lemma \ref{lem:properties_of_directions} (iii). Moreover, a straightforward computation shows
\begin{align*}
Q(z)&=\frac{1}{n}\tr M(z)+\lamax(M(z)^\circ)\\
&=\frac{\mu_+-\mu}{\mu_+-\mu_-}T_-(z)+\frac{\mu-\mu_-}{\mu_+-\mu_-}T_+(z)+\lamax(M(z)^\circ)
\end{align*}
and thus by Lemma \ref{lem:properties_of_directions} (iii) we have
\begin{align*}
Q(z+t\tilde{z}(z))=Q(z)+t\frac{T_+(z)-T_-(z)}{\mu_+-\mu_-}.
\end{align*}
For $\abs{t}\abs{T_+(z)-T_-(z)}\leq (e-Q(z))(\mu_+-\mu_-)$ and $\abs{t}< \dist(\mu,\set{\mu_-,\mu_+})$ we therefore conclude $z+t\tilde{z}(z)\in \overline{U}_0$.

From now on we consider the remaining case $Q(z)=e$ and $T_+(z)\neq T_-(z)$. Note that then necessarily $\lamin (M(z))<e$, because otherwise $e=\lamax (M(z))=\lamin (M(z))$ yields $M(z)^\circ=0$ and thus 
\[
e=Q(z)=\frac{\mu_+-\mu}{\mu_+-\mu_-}T_-(z)+\frac{\mu-\mu_-}{\mu_+-\mu_-}T_+(z).
\]
Since $T_+(z)\leq e$, $T_-(z)\leq e$ this equality can only hold if $T_+(z)=T_-(z)=e$, which is excluded in the case we are considering.

Let us assume $T_-(z)> T_+(z)$, the other case follows similarly. We consider Euler directions of the form $\bar{z}=(0,\bar{w},\mu_+\bar{w},\bar{\sigma},\bar{q})$, where $\bar{w}\in\R^n$ and $\bar{\sigma}\in  \mathcal S_0^{n\times n}$ will be chosen later and $\bar{q}=\bar{q}(\bar{w},\bar{\sigma})$ by Lemma \ref{lem:properties_of_directions} (i). These Euler directions allow us to preserve $T_-$ due to Lemma \ref{lem:properties_of_directions} (iv), i.e., $T_-(z+t\bar{z}_+)=T_-(z)\leq e$ for all $t\in\R$.

Now we need to guarantee that $Q(z+t\bar{z})=Q(z)=e$ for small enough $\abs{t}$ and some choice of $\bar{w}$, $\bar{\sigma}$. As in the cases $z\in K_\pm'\setminus K$ we can again assume that the matrix $M(z)$ is diagonal with entries $e=\lambda_1\geq \lambda_2\geq \ldots\geq \lambda_n$ and $\lambda_n<e$. As before we take $\bar{w}=e_n$, $\bar{m}=\mu_+e_n$ and 
the uniquely determined pair $(\bar{\sigma},\alpha)\in  \mathcal S_0^{n\times n}\times\R$ satisfying
\begin{align*}
M(z+t\bar{z})&=M(z)+\alpha t e_n\otimes e_n+\frac{\mu_+(\mu_+-\mu_-)}{\mu-\mu_-}t^2e_n\otimes e_n.
\end{align*}
For small enough $\abs{t}$ we therefore conclude that this Euler perturbation does not affect the maximal eigenvalue, i.e., $Q(z+t\bar{z})=Q(z)=e$ for $\abs{t}$ small. 

Furthermore, the last condition needed for $z+t\bar{z}\in\overline{U}$ simply follows by the continuity of $T_+$, i.e., for all $\abs{t}$ small enough there holds
\[
T_+(z+t\bar{z})< T_-(z)\leq e.
\]
\end{proof}
Now we have all ingredients for the proof of $K^\Lambda=K^{co}=\overline{U}$ at hand.
\begin{proof}[Proof of Proposition \ref{prop:lambda_convex_hull}]
Lemma \ref{lem:U_convex_closure_of_U} implies $K^\Lambda\subset K^{co}\subset \overline{U}$, while Lemma \ref{lem:boundedness_of_U_extreme_points} says that the $\Lambda$-extreme points of the up to the $q$-component compact set $\overline{U}$ are contained in $K$. The inclusion $\overline{U}\subset K^\Lambda$ follows by the Krein-Milman theorem for $\Lambda$-convex sets, cf. \cite{Kircheim}, Lemma 4.16.
\end{proof}
\subsection{Perturbing along sufficiently long enough segments}

In this subsection we prove that any element from $U$ is contained in a sufficiently long admissible line segment, similarly to Section 4.3 from \cite{DeL-Sz-Adm}. We recall the projection operator $\pi$ defined in \eqref{eq:projection}.
We have the following result.
\begin{lemma}\label{lem:segments}
For any $z\in U$ there exists $\bar{z}\in\Lambda$ such that we have
\begin{align*}
[z-\bar{z},z+\bar{z}]\subset U\ \text{and } |\pi(\bar{z})|\geq \frac{1}{2N} d(\pi(z),K/\sim),
\end{align*}
where $N=\dim(Z)$ and $d$ denotes the Euclidian distance on $\pi(Z)$.
\end{lemma}
\begin{proof}
We proceed as in the proof of Lemma 4.7 from \cite{DeL-Sz-Adm}. Since $z\in U=\text{int}K^{co}$, it follows from Carath\'eodory's theorem that it lies in the interior of a simplex in $Z$ spanned by $K$, i.e. there exist $\lambda_i\in(0,1)$, $z_i\in K$, $i\in{1,\ldots,N+1}$, $\sum_i \lambda_i=1$,  such that
$$z=\sum_{i=1}^{N+1}\lambda_i z_i.$$
We may also assume that the coefficients are ordered such that $\lambda_1=\max_i{\lambda_i}$, then for any $j>0$ we have
$$z\pm \frac{1}{2}\lambda_j(z_j-z_1)\in \text{int}K^{co}.$$
Indeed, one may rewrite $z\pm \frac{1}{2}\lambda_j(z_j-z_1)=\sum_i \kappa_i z_i$, where $\kappa_1=\lambda_1\mp\frac{1}{2}\lambda_j$, $\kappa_j=\lambda_j\pm\frac{1}{2}\lambda_j$ and $\kappa_i=\lambda_i$ for $i\not\in\{1,j\}$, such that these coefficients are in $(0,1).$

Furthermore, since we have $z-z_1=\sum_{i=2}^{N+1}\lambda_i(z_i-z_1)$, it follows that
\begin{align}\label{optmw}
|\pi(z)-\pi(z_1)|\leq N \max_{i=2,\ldots,N+1} \lambda_i |\pi(z_i)-\pi(z_1)|.
\end{align}
Choose $j>0$ such that $\max_{i=2,\ldots,N+1} \lambda_i |\pi(z_i)-\pi(z_1)|=\lambda_j|\pi(z_j)-\pi(z_1)|$, and let 
$$\bar{z}=\frac{1}{2}\lambda_j(z_j-z_1).$$
Then $[z-\bar{z},z+\bar{z}]\subset\text{int} K^{co}$
and  
$$d(\pi(z),K/\sim)\leq|\pi(z)-\pi(z_1)|\leq 2N |\pi(\bar{z})|.$$ 

To conclude the proof of the lemma, it would suffice to have $\bar{z}\in\Lambda$.
While this in general may not be true a priori, we know from Corollary \ref{bigenoughcone} that it is true up to changing the pressure in $\bar{z}$. However, since the constraints in $K$, respectively the inequalities in $U$ do not involve the pressure, this can be done such that $z\pm\bar{z}\in\text{int} K^{co}$ still remains valid. This concludes the proof.
\end{proof}
\subsection{Continuity of constraints}

We now go back to the $(x,t)\in\mathscr{D}$ dependent sets of constraints $K_{(x,t)}$ defined in \eqref{eq:nonlinear_constraints_Kxt}.
We have the following result regarding the continuity of the nonlinear constraints in \eqref{eq:constraints}, given the continuity of the associated energy. This will allow us to have a set of subsolutions which is bounded in $L^2(\mathscr{D})$. 

\begin{lemma}\label{lem:cont_of_constr}
Let $\mathscr{U}\subset\mathscr{D}$ be an open, bounded set and $e:\Omega\times[0,T)\rightarrow\R_+$. If the map $(x,t)\mapsto e(x-1/2gt^2e_n,t)\in\R_+$ is continuous and bounded on $\mathscr{U}$, then it follows that the map $(x,t)\mapsto K_{(x,t)}/\sim$ is continuous and bounded on $\mathscr{U}$ with respect to the Hausdorff metric $d_{\mathcal{H}}$.
\end{lemma}

The proof of Lemma \ref{lem:cont_of_constr} is based on the following observation, which can be found in \cite{Crippa} as Lemma 3.1.
\begin{lemma}\label{lem:haus}
Suppose $A,B\subset\mathbb{R}^l$ for some $l\in\N$ are compact sets and $r>0$ such that
\begin{itemize}
\item for any $z\in A$ there exists $z'\in B\cap B_r(z),$
\item for any $z\in B$ there exists $z'\in A\cap B_r(z).$
\end{itemize}
Then $d_{\mathcal{H}}(A,B)< r$.
\end{lemma}
\begin{proof}
See \cite{Crippa}.
\end{proof}

\begin{proof}[Proof of Lemma \ref{lem:cont_of_constr}]
Fix $y=(x,t)\in\mathscr{U}$. For $\varepsilon>0$ there exists $\delta>0$ such that 
\begin{align}\label{eq:squareroot}
\abs{e(y)-e(y')}<\varepsilon\quad\text{and}\quad\left|\left(n\frac{e(y)}{\mu_\pm}\right)^{1/2}-\left(n\frac{e(y')}{\mu_\pm}\right)^{1/2} \right|<\varepsilon,
\end{align} for any $y'\in B_\delta(y)\subset\mathscr{U}.$ Using Lemma \ref{lem:haus} we will prove $d_{\mathcal{H}}( K_y/\sim, K_{y'}/\sim)< c \varepsilon$ for any $y'\in B_\delta(y)\subset\mathscr{U}$, with $c>0$ depending only on $\mu_+$, $\mu_-$ and $n$.

Let $$z=(\mu,w,\mu w,\mu w\otimes w-e(y)\id,q)\in K_y,$$
with $\mu\in\{\mu_+,\mu_-\}$ and $\mu|w|^2=ne(y)$. It follows 
that
$$w=\left(\frac{n}{\mu}e(y)\right)^{1/2}b,$$
for some $b\in S^{n-1}$.

We define $$z'=(\mu,w',\mu w', \mu w'\otimes w'-e(y')\id,q)$$ by setting 
\begin{align*}
w'=\left(\frac{n}{\mu}e(y')\right)^{1/2}b.
\end{align*}
Note that $z'\in K_{y'}$.

Furthermore, from \eqref{eq:squareroot} it follows that
$$|w-w'|<\varepsilon,$$
from which one can conclude $\abs{z-z'}<c\varepsilon$ for some $c=c(\mu_\pm,n)>0$.

Due to the symmetry of this construction, one can similarly prove that for any $z'\in K_{y'}$ there exists $z\in K_y$ such that $|z-z'|<c\varepsilon.$ The result then follows from Lemma \ref{lem:haus}.

The boundedness of $\bigcup_{y\in\mathscr{U}}K_y$ follows from Lemma \ref{lem:boundedness_of_U_extreme_points} and the assumption that the function $e$ is bounded.
\end{proof}

\section{Proof of Theorem \ref{thm:main2}}\label{sec:conclusion}

In this section we conclude the proof of Theorem \ref{thm:main2} by using the Baire category method.

\subsection{The Baire category method}

We introduce the notion of subsolution associated with \eqref{eq:linear_system},\eqref{eq:nonlinear_constraints_Kxt}. For simplicity of notation, in this subsection we will, as in the proof of Lemma \ref{lem:cont_of_constr}, denote $y:=(x,t)$.
\begin{definition}\label{def:subsol}
We say that $z:\mathscr{D}\to Z$ is a subsolution of \eqref{eq:linear_system} associated with the set of constraints $K_y$, iff it is a weak solution of \eqref{eq:linear_system} in the sense of Definition \ref{def:translinweak}  in $\mathscr{D}$, $\pi(z)$ is continuous in $\mathscr{U}$, $z(y)\in K_y$  holds for a.e. $y\in\mathscr{D}\setminus\mathscr{U}$ and 
\begin{align}\label{eq:tartar_sub}
z(y)\in U_y=\text{int}K_y^{co}\text{ for any }y\in\mathscr{U}.
\end{align}
\end{definition}

We have the following convex integration result.
\begin{theorem}\label{thm:tartar}
Suppose that there exists a subsolution $z_0$ in the sense of Definition \ref{def:subsol}. Then there exist infinitely many weak solutions $z:\mathscr{D}\to Z$
of \eqref{eq:linear_system}  which coincide with $z_0$ a.e. in $\mathscr{D}\setminus\mathscr{U}$, satisfy $z(y)\in K_y$ a.e. in $\mathscr{D}$, and for every open ball $B\subset\mathscr{U}$ the solutions satisfy the mixing property \begin{align}\label{eq:mixing_property2}
\int_B\mu_+-\mu(x,t)\:d(x,t)\cdot\int_B\mu(x,t)-\mu_-\:d(x,t)>0.
\end{align} 
Furthermore, among these weak solutions there exists a sequence $\{z_k\}_{k\geq 1}$ such that $\pi(z_k)$ converges weakly to $\pi(z_0)$ in $L^2(\mathscr{U};\pi(Z))$.
\end{theorem}

%\todo{comment about initial data?}

%
The proof is similar to those in \cite{Crippa,Sz-Muskat}, the only main difference being that one has to carefully track the role of the projection $\pi$. However, since the existence of the pressure is implicit in Definition \ref{def:translinweak} due to the use of divergence-free test functions, this can be done without any serious difficulty.

The main building block of the proof is the following perturbation lemma.
\begin{lemma}\label{lem:tartarpert}
Suppose that there exists a subsolution $z$ such that
$$\int_\mathscr{U} d(\pi(z(y)),K_y/\sim)^2\, dy = \varepsilon>0.$$
Then there exist $\delta=\delta(\varepsilon)>0$ and a sequence of subsolutions $\{z_k\}_{k\geq 0}$ such that
\begin{itemize}
\item $z_k=z$ in $\mathscr{D}\setminus\mathscr{U}$, for any $k\geq 0$,
\item $\int_\mathscr{U} |\pi(z_k(y)-z(y))|^2\, dy \geq \delta,$ for any $k\geq 0$,
\item $\pi(z_k)\rightharpoonup \pi(z)$ in $L^2(\mathscr{U};\pi(Z))$ as $k\to+\infty$.
\end{itemize}
\end{lemma}
To prove Lemma \ref{lem:tartarpert}, we will use the following result which can be found together with its proof as Lemma 2.1 in \cite{Crippa}.
\begin{lemma}\label{lem:Crippa1}
Let $K\subset\R^n$
be a compact set. Then for any compact set $C\subset\text{int}K^{co}$ there exists
$\varepsilon>0$ such that for any compact set $K'\subset\R^n$ with $d_\mathcal{H}(K,K')<\varepsilon$ we have
$C\subset\text{int}(K')^{co}$.
\end{lemma}
\begin{proof}[Proof of Lemma \ref{lem:tartarpert}]
Fix $y\in\mathscr{U}$. 
From Lemma \ref{lem:segments} it follows that there exists some $C'>0$ independent of $y$ and $z$, and some
$\bar{z}(y)\in\Lambda$ such that
\begin{align*}
[z(y)-\bar{z}(y),z(y)+\bar{z}(y)]\subset U_y,\quad
|\pi(\bar{z}(y))|\geq C' d(\pi(z(y)),K_y/\sim).
\end{align*}
Now Lemma \ref{lem:cont_of_constr}, the continuity of $\pi(z)$ and Lemma \ref{lem:Crippa1} applied to the projected sets imply
%Furthermore, using Lemma \ref{lem:Crippa1}, (H3) and the continuity of $z$, it follows 
that there exist $r(y),R(y)>0$ such that\begin{align*}
[z(y')-\bar{z}(y),z(y')+\bar{z}(y)]+\overline{B_{R(y)}(0)}\subset U_{y'},\\
d(\pi(z(y')),K_{y'}/\sim) \leq 2 d(\pi(z(y)),K_y/\sim),\end{align*}
for any $y'\in B_{r(y)}(y).$

Using Lemma \ref{lem:locpw}, we find a sequence $\{z_{y,N}\}_{N\geq 0}\subset C^\infty_c(B_1(0))$ solving \eqref{eq:linear_system} such that
\begin{itemize}
\item $z_{y,N}(y')\in[-\bar{z}(y),\bar{z}(y)]+\overline{B_{R(y)}(0)}$ for all $y'\in B_1(0)$, $N\geq 0,$
\item $z_{y,N}\rightharpoonup 0$ in $L^2$,
\item $\int |\pi(z_{y,N})|^2 \, d\tilde{y}\geq C|\pi(\bar{z}(y))|^2$ for all $N\geq 0.$
\end{itemize}
From here on the proof is the same as Step 2 of the proof of Lemma 2.4 from \cite{Crippa},
using a standard covering argument, therefore the details are left to the reader.
\end{proof}
\begin{proof}[Proof of Theorem \ref{thm:tartar}]
Let 
\begin{align*}
X_0=\left\{z'\in L^2(\mathscr{D};\pi(Z))\text{ such that }z'=\pi(z)\text{ for some  subsolution } z \text{ in the sense}\right.\\ \left.\text{of Definition \ref{def:subsol} satisfying }z=z_0\text{ on }\mathscr{D}\setminus\mathscr{U}\right\},
\end{align*}
and $X$ denote the closure of $X_0$ with respect to the weak $L^2$ topology. From Lemma \ref{lem:cont_of_constr} it follows that $X_0$ is bounded, therefore $X$ is metrizable, denote its metric by $d_X(\cdot,\cdot)$.
Also since the existence of the pressure is implicit in Definition \ref{def:translinweak} due to the use of divergence-free test functions, it follows that for any $z'\in X$ there exists a possibly distributional pressure $q'$ such that $(z',q')$ is indeed a weak solution of \eqref{eq:linear_system}.

We observe that the functional $I(z')=\int_\mathscr{U} |z'|^2\, dy$ is a Baire-1 function on $X$. Indeed, setting
$$
I_j(z')=\int_{\set{y\in \mathscr{U}: d(y,\partial\mathscr{U})>1/j}} |(z'*\chi_j)(y)|^2\ dy,
$$
where $\chi_j\in C_c^\infty(B_{1/j}(0))$ is the standard mollifying sequence, one obtains that $I_j$ is continuous on $X$ and that $I_j(z')\to I(z')$ as $j\to+\infty.$

It follows from the Baire category theorem that the set 
$$Y=\{z'\in X:\ I\text{ is continuous at }z'\}$$
is residual in $X$. We claim that for any $z'\in Y$ it follows that $$J(z'):=\int_\mathscr{U} d(z'(y),K_y/\sim)^2\, dy=0.$$

Suppose the contrary, then $J(z')=\varepsilon>0$ for some $z'\in Y$, and let $z'_j\in X_0$ be a sequence which converges to $z'$ w.r.t. $d_X$. Since $I$ is continuous at $z'$, it follows that $z'_j\to z'$ strongly in $L^2(\mathscr{U};\pi(Z))$. Note that $J$ is continuous with respect to the strong $L^2$-topology. Therefore we may assume that $J(z'_j)>\varepsilon/2$ for all $z'_j$.

Since $z'_j\in X_0$, there exists some $z_j:\mathscr{D}\to Z$ which is a subsolution in the sense of Definition \ref{def:subsol} and such that $z'_j=\pi(z_j).$ We may then apply Lemma \ref{lem:tartarpert} to deduce that there exists $\delta=\delta(\varepsilon)>0$ and a subsolution $\tilde{z}_j$ such that $\int_\mathscr{U} |\pi(z_j(y)-\tilde{z}_j(y))|^2\, dy \geq \delta$ and $\pi(z_j-\tilde{z}_j)\rightharpoonup 0$ weakly in $L^2$. Since $z'_j=\pi(z_j)\to z'$ and $z'\in Y$, we conclude as before $\pi(\tilde{z}_j)\to z'$ strongly in $L^2$  contradicting the fact that $\pi(\tilde{z}_j)$ and $z_j'$ are uniformly bounded away from each other.
We thus have showed that the set of solutions $J^{-1}(0)$ is residual in $X$.

The proof of the mixing property \eqref{eq:mixing_property2} follows by another application of the Baire category theorem and is exactly the same as in \cite{Castro_Cordoba_Faraco}. For convenience we briefly present it here as well. Let $B$ be an open ball contained in $\mathscr{U}$. The set
\[
X_B^{\mu_\pm}=\set{z'\in X:\int_B\mu_\pm-\mu(x,t)\:d(x,t)=0}
\]
is closed in $X$ and has empty interior, since $X_B^{\mu_\pm}\subset X\setminus X_0$. Therefore $J^{-1}(0)\setminus X_B^{\mu_\pm}$ is residual in $X$, as is $J^{-1}(0)\setminus\left(\bigcup_i\big(X_{B_i}^{\mu_+}\cup X_{B_i}^{\mu_-}\big)\right)$ for any countable union of balls $B_i\subset\mathscr{U}$. By taking all balls $(B_i)_{i\in\N}\subset\mathscr{U}$ with rational centers and radii we can conclude the statement. 
\end{proof}

\subsection{Conclusion}
In order to prove our convex integration result for \eqref{eq:main} we apply a transformation similar to \eqref{eq:transf} to the differential inclusion \eqref{eq:linear_system},\eqref{eq:nonlinear_constraints_Kxt} and in particular also its relaxation. Recall from Section \ref{sec:reformulation_as_differential_inclusion} that for a bounded domain $\Omega\subset\R^n$ and $T>0$ we defined $\mathscr{D}=\set{(y,t)\in\R^n\times(0,T):y-\frac{1}{2}gt^2e_n\in\Omega}$.

Now let $z=(\mu,w,m,\sigma,q)$ be a weak solution of \eqref{eq:linear_system} with some suitable initial data. Defining again $y:=x+\frac{1}{2}gt^2e_n$, as well as
\begin{align}\label{eq:transflin}
\begin{split}
\rho(x,t)&=\mu\left(y,t\right),\\
v(x,t)&=w\left(y,t\right)-g t e_n,\\
u(x,t)&=m\left(y,t\right)-\mu\left(y,t\right)g t e_n,\\
P(x,t)&=q\left(y,t\right)+gt\frac{1}{n}\left(gt\mu(y,t)-2m_n\left(y,t\right)\right),\\
S(x,t)&=\sigma\left(y,t\right)-gt
\left(m\left(y,t\right)\otimes e_n+e_n\otimes m\left(y,t\right) \right)\\&
\phantom{asdasdasd}+g^2t^2\mu(y,t)e_n\otimes e_n
-\left(gt\frac{1}{n}\left(gt\mu(y,t)-2m_n\left(y,t\right)\right)\right)\id,
\end{split}
\end{align}
one obtains through lenghty but straightforward calculations that $(\rho,v,u,S,P)$ is a weak solution of \eqref{eq:linear_system_retransformed}
with the same initial data. Also here the transformation can be inverted in an obvious way, mapping a solution of \eqref{eq:linear_system_retransformed} to a solution of \eqref{eq:linear_system}. 

Furthermore, for a given function $e:\Omega\times[0,T)\rightarrow\R_+$ the condition $z(y,t)\in K_{(y,t)}$ for $y=x+\frac{1}{2}gt^2$, $(x,t)\in\Omega\times(0,T)$ and with $K_{(y,t)}$ defined in \eqref{eq:nonlinear_constraints_Kxt} 
translates to 
$(\rho,v,u,S,P)(x,t)\in \cK_{(x,t)}$ with $\cK_{(x,t)}$ defined in \eqref{eq:definition_of_nonlinear_constraints}. Similarly if we define $U_{(y,t)}$ to be the interior of the convex hull of $K_{(y,t)}$ then by Proposition \ref{prop:lambda_convex_hull} the condition $z(y,t)\in U_{(y,t)}$ translates to $(\rho,v,u,S,P)(x,t)\in\cU_{(x,t)}$ with $\cU_{(x,t)}$ defined in \eqref{eq:conditions_for_being_in_tx_dependent_hull2},\eqref{eq:conditions_for_being_in_tx_dependent_hull}. Since the transformation is an affine bijection, we also see that $\cU_{(x,t)}$ is the interior of the convex hull of $\cK_{(x,t)}$.

We have now all pieces together to prove our main result.

\begin{proof}[Proof of Theorem \ref{thm:main2}] 
Let $z=(\rho,v,u,S,P):\Omega\times (0,T)\to Z$  be a subsolution (in the sense of Definition \ref{def:subsolEuler}) of \eqref{eq:main} associated with $e:\Omega\times [0,T)\rightarrow\R_+$ bounded and initial data $(\rho_0,v_0)\in L^\infty(\Omega)\times L^2(\Omega;\R^n)$ satisfying \eqref{eq:id_comp}. We 
%set as before $$\mathscr{D}=\set{(x,t)\in\R^n\times(0,T):x-\frac{1}{2}gt^2e_n\in\Omega}$$ and define also 
also define the transformed mixing zone 
\[
\mathscr{U}'=\set{(y,t)\in\R^n\times (0,T):\left(y-\frac{1}{2}gt^2e_n,t\right)\in\mathscr{U}}.
\]
The inverse of the transformation \eqref{eq:transflin} applied to $z$ gives us a weak solution of \eqref{eq:linear_system} (in the sense of Definition \ref{def:translinweak}) which we call $z'=(\mu,w,m,\sigma,q):\mathscr{D}\to Z$. 
By the discussion of this section and Definition \ref{def:subsolEuler}, $\pi(z')$ is continuous on $\mathscr{U}'$, $z'(y,t)\in U_{(y,t)}=\text{int} K_{(y,t)}^{co}$ for all $(y,t)\in \mathscr{U}'$ and $z'(y,t)\in K_{(y,t)}$ for a.e. $(y,t)\in\mathscr{D}\setminus \mathscr{U}'$. 

In other words $z'$ is a subsolution of the differential inclusion \eqref{eq:linear_system}, \eqref{eq:nonlinear_constraints_Kxt} in the sense of Definition \ref{def:subsol} (with mixing zone $\mathscr{U}'$).  Theorem \ref{thm:tartar} therefore provides us with infinitely many solutions of our differential inclusion \eqref{eq:linear_system}, \eqref{eq:nonlinear_constraints_Kxt} which outside of $\mathscr{U}'$ agree with $z'$ and inside $\mathscr{U}'$ satisfy the mixing property \eqref{eq:mixing_property2}, 
as well as with a sequence of solutions such that $\pi(z'_k)$ converges $L^2$-weakly to $\pi(z')$. 

One may then transfer these conclusions to the setting of Theorem \ref{thm:main2} via Lemma \ref{lem:equivalence_between_pde_and_differential_inclusion}.

Let us now briefly explain how to establish the admissibility of the obtained solutions, provided that $\pi(z)$ is in addition of class $\cC^0([0,T];L^2(\Omega;\pi(Z)))$. As before let $z'$ be the corresponding transformed subsolution defined on $\mathscr{D}$.
Due to an improvement of the Tartar framework as in \cite{Castro_Faraco_Mengual,DeL-Sz-Adm} one can show that the induced sequence $\{\pi(z'_k)\}_{k\in\N}$ not only converges weakly in $L^2(\mathscr{D})$ to $\pi(z')$, but weakly on every time-slice $\mathscr{D}(t)$ uniformly in $t\in[0,T]$. It is in fact straightforward but quite lengthy to adapt the proof from \cite{DeL-Sz-Adm} to our situation, therefore we omit the details, cf. also \cite{Castro_Faraco_Mengual} and in particular Remark 2.3 therein. Transforming $z_k'$ again to $z_k$ we conclude that the associated energies
\begin{align*}
E_k(x,t):=\frac{n}{2} e(x,t)- gt e_n\cdot u_k(x,t)-\frac{1}{2}\rho_k (x,t)g^2t^2+\rho_k(x,t) gx_n
\end{align*}
 satisfy
\[
\int_\Omega E_k(x,t)\:dx\rightarrow \int_\Omega E_{sub}(x,t)\:dx
\]
uniformly in $t\in[0,T]$ as $k\rightarrow\infty$,
recall the definitions \eqref{eq:relation_between_E_and_e}, \eqref{eq:subsolenergy}.

However this does not yet allow us to conclude the admissibility of the induced solutions, since the difference 
\[
\varepsilon(t):=\int_\Omega E_{sub}(x,0)-E_{sub}(x,t)\:dx>0,\quad t\in(0,T)
\] 
goes to $0$ as $t\searrow 0$. Nonetheless, similarly to \cite[Definition 2.4]{Castro_Faraco_Mengual} (but a lot less technical for our purposes) we can extend  the definiton of the space $X_0$, such that the sequence (or any solution obtained by the convex integration scheme) satisfies
\begin{align*}
\left| \int_\Omega  gt e_n\cdot (u(x,t)-u_k(x,t))+\left(\frac{1}{2}g^2t^2- gx_n\right)(\rho(x,t)-\rho_k (x,t))\, dx \right|  \leq \varepsilon(t),
\end{align*}
for all $t\in[0,T],\ k\geq 0$. The statement follows.
\end{proof}

\section{Subsolutions}\label{sec:subsolutions}

We now turn to the construction of Rayleigh-Taylor subsolutions. We start by observing that the relaxation inside the mixing zone $\mathscr{U}\subset \Omega\times(0,T)$ given in Definition \ref{def:subsolEuler} 
can be equivalently rewritten (in the spirit of \cite{Castro_Cordoba_Faraco}) as the system
\begin{align}\label{eq:transfss}
\begin{split}
\partial_t (\rho v+f)+\text{div } S + \nabla P &= - \rho g e_n,\\
\text{div } v&=0,\\
\partial_t \rho + \text{div }(\rho v+f)&=0,
\end{split}
\end{align}
where 
\begin{align*}
f:=&\frac{\rho_+-\rho}{\rho_+-\rho_-}\sqrt{\frac{ne}{\rho_+}}(\rho-\rho_-)\xi+\frac{\rho-\rho_-}{\rho_+-\rho_-}\sqrt{\frac{ne}{\rho_-}}(\rho_+-\rho)\eta,
\end{align*}
for some functions $\xi,\eta:\mathscr{U}\to\R^n$
satisfying
\begin{align}
\sqrt{ne}\left(\frac{\rho-\rho_-}{\sqrt{\rho_+}}\xi-\frac{\rho_+-\rho}{\sqrt{\rho_-}}\eta\right)=(\rho_+-\rho_-)(v+gte_n),\quad |\xi|< 1,\quad |\eta|< 1\label{eq:euler_reynolds_condition1}
\end{align}
in $\mathscr{U}$.
The condition on $\lambda_{\max}(A(z))$ from \eqref{eq:conditions_for_being_in_tx_dependent_hull} with $u$ replaced by $\rho v+f$ is kept in accordance with Definition \ref{def:subsolEuler}.

Indeed, in order to see this, given a subsolution $z=(\rho,v,u,S,P)$ it suffices to set 
\begin{align}\label{eq:setavg}
\xi:=\sqrt{\frac{\rho_+}{ne}}\frac{u-\rho_-v+(\rho-\rho_-)gte_n}{\rho-\rho_-},\quad
\eta:=\sqrt{\frac{\rho_-}{ne}}\frac{u-\rho_+v+(\rho-\rho_+)gte_n}{\rho_+-\rho}.
\end{align}
Conversely, given $f$, it suffices to set $u:=\rho v+f$ to obtain a subsolution in the sense of Definition \ref{def:subsolEuler}.

%In the same style one could also rewrite inequality \eqref{eq:conditions_for_being_in_tx_dependent_hull} and the first equation of \eqref{eq:transfss} as an evolution equation of Euler-Reynolds type, but since this is neither relevant to the construction of our subsolutions nor particularly elegant, we omit the reformulation at this point. \todo{leave or cut last comment}

\begin{proof}[Proof of Theorem \ref{thm:rayleigh_taylor_subsolutions}] Now let $n=2$, $T>0$ and $\Omega\subset \R^2$ be the rectangle stated in the Theorem. In view of the equivalent reformulation above our goal is to find a suitable combination of functions $\xi,\eta$ and $e$, such that \eqref{eq:transfss} has a solution satisfying the energy inequality \eqref{eq:weakadm} in a strict sense.

In fact we will look for one-dimensional solutions of \eqref{eq:transfss}, i.e. a subsolution $z$ in the sense of Definition \ref{def:subsolEuler}, which is independent of $x_1$ and satisfies $u=u_2e_2$, $\xi=\xi_2e_2$, $\eta=\eta_2e_2$ respectively. We further assume $v\equiv 0$.

If we have choosen $\xi$, $\eta$, then condition \eqref{eq:euler_reynolds_condition1} implies that $e$ in the mixing zone is determined by
\begin{align}\label{eq:subsolen}
\sqrt{2e}=\frac{\sqrt{\rho_-\rho_+}(\rho_+-\rho_-)gt}{\sqrt{\rho_-}(\rho-\rho_-)\xi_2-\sqrt{\rho_+}(\rho_+-\rho)\eta_2}.
\end{align}
Note also that under condition \eqref{eq:euler_reynolds_condition1} the denominator will always be positive for $t>0$. Outside the mixing zone we will have $e=\frac{1}{2}\rho g^2t^2$ in accordance with \eqref{eq:relation_between_E_and_e}.

The last equation in \eqref{eq:transfss} then becomes
\begin{align}
\partial_t \rho+ gt \partial_{x_2}\left(\frac{(\rho_+-\rho)(\rho-\rho_-)(\sqrt{\rho_-}\xi_2+\sqrt{\rho_+}\eta_2)}{(\rho-\rho_-)\sqrt{\rho_-}\xi_2-(\rho_+-\rho)\sqrt{\rho_+}\eta_2} \right)=0.
\end{align} 
Using the change of coordinates $\rho(x,t)=y(x,gt^2/2)$ and interpreting the $\xi_2,\eta_2$ as functions of $\rho$ only, one obtains equivalently
\begin{align}\label{eq:entr}
\partial_t y+ \partial_{x_2}G(y)=0,
\end{align} 
with
$$G(y)=\frac{(\rho_+-y)(y-\rho_-)(\sqrt{\rho_-}\xi_2(y)+\sqrt{\rho_+}\eta_2(y))}{(y-\rho_-)\sqrt{\rho_-}\xi_2(y)-(\rho_+-y)\sqrt{\rho_+}\eta_2(y)}.$$

Now if $G:[\rho_-,\rho_+]\to\R$ is uniformly strictly convex, then one may consider the unique entropy solution (cf. Section 3.4.4 in \cite{evans}) of \eqref{eq:entr} with Rayleigh-Taylor initial data $\rho_0$ to obtain that 
\begin{align}\label{eq:entropy_solution}
\begin{split}
\rho(x_2,t)= \left\{
\begin{array}{lll}
 \rho_-,\text{ when }x_2\leq\frac{1}{2}gt^2 G'(\rho_-),\\
 (G')^{-1}\left(\frac{2x_2}{gt^2}\right),\text{ when }x_2\in\left(\frac{1}{2}gt^2 G'(\rho_-),\frac{1}{2}gt^2 G'(\rho_+)\right),\\
 \rho_+,\text{ when }x_2\geq\frac{1}{2}gt^2 G'(\rho_+).
\end{array} 
\right.
\end{split}
\end{align}
Observe that this already implies that the height of the mixing zone grows (up to a constant) like $\frac{1}{2}gt^2$, more precisely we will have 
\begin{align}\label{eq:entreqq}
\mathscr{U}=\set{(x,t)\in\Omega\times(0,T):\frac{1}{2}gt^2G'(\rho_-)<x_2<\frac{1}{2}gt^2 G'(\rho_+)}.
\end{align}

It is easy to check that if one is able to choose $\xi_2,\eta_2\in(-1,1)$ such that $G$ is indeed uniformly strictly convex and the above entropy solution exists, then defining 
\begin{gather*}
u_2(x_2,t):=gt G(\rho(x_2,t)),\quad
S:=\frac{(\rho_++\rho_--\rho)u_2^2}{2(\rho_+-\rho)(\rho-\rho_-)}\begin{pmatrix}
-1 & 0\\
0 & 1
\end{pmatrix},\\
P(x_2,t):=S_1(x_2,t)-\int_{\frac{1}{2}gt^2 G'(\rho_-)}^{x_2}\partial_t u_2(x',t)-\rho(x',t)g\:dx_2,
\end{gather*}
with $u_2$ and $S$ extended by $0$ outside $\mathscr{U}$,
one truly obtains a subsolution in the sense of Definition \ref{def:subsolEuler}. Indeed the relaxed momentum equation holds by definition of $P$ and inequality \eqref{eq:conditions_for_being_in_tx_dependent_hull} reduces to 
\begin{align*}
e&>\frac{(\rho_++\rho_--\rho)u_2^2}{2(\rho_+-\rho)(\rho-\rho_-)}+gt u_2+\frac{1}{2}\rho g^2t^2\\
&=\frac{\rho_+-\rho}{\rho_+-\rho_-}\frac{\rho_-}{2}\left(\frac{u_2}{\rho-\rho_+}+gt\right)^2+\frac{\rho-\rho_-}{\rho_+-\rho_-}\frac{\rho_+}{2}\left(\frac{u_2}{\rho-\rho_-}+gt\right)^2,
\end{align*}
which holds, since by our reformulation inequalities \eqref{eq:conditions_for_being_in_tx_dependent_hull2} are automatically satisfied for $\xi_2,\eta_2\in(-1,1)$ and $e$ defined in \eqref{eq:subsolen}.

Therefore, all that remains to do in order to finish the construction of RT-subsolutions is to find $\xi_2,\eta_2:(\rho_-,\rho_+)\rightarrow(-1,1)$ such that $G$ is uniformly strictly convex and to assure the admissibility \eqref{eq:weakadm} (in a strict sense for $t>0$) of the associated total energy \eqref{eq:subsolenergy}.

Denoting
$$Q(\rho):=(\rho-\rho_-)\sqrt{\rho_-}\xi_2(\rho)-(\rho_+-\rho)\sqrt{\rho_+}\eta_2(\rho)>0,$$
one has $e(x_2,t)=g^2t^2\tilde{e}(\rho(x_2,t))$ with $$\tilde{e}(\rho):=\frac{1}{2}\frac{\rho_+\rho_-(\rho_+-\rho_-)^2}{Q(\rho)^2}.$$
By the transformation $x_2=\frac{1}{2}gt^2G'(\rho)$ the desired admissibility \eqref{eq:weakadm} in the strict sense is then equivalent to
\begin{equation}\label{eq:equivalent_subsol_admissibility}
\int_{\rho_-}^{\rho_+}\left(\tilde{e}(\rho)-\frac{1}{2}\rho-G(\rho)\right)G''(\rho)\:d\rho<\frac{1}{4}\int_{\rho_-}^{\rho_+}\big(\rho_0(G'(\rho))-\rho\big)\big(G'(\rho)^2\big)'\:d\rho.
\end{equation}

We further make the ansatz
 $\tilde{e}(\rho_\pm)=\frac{1}{2}\rho_\pm$, in other words that $e$ is continuious across $\partial\mathscr{U}$. 
Then partial integration shows that \eqref{eq:equivalent_subsol_admissibility} is equivalent to
\begin{equation}\label{eq:equivalent_subsol_adm_2}
I_{\xi_2,\eta_2}:=\int_{\rho_-}^{\rho_+}\left(\tilde{e}'(\rho)-\frac{3}{4}G'(\rho)\right)G'(\rho)\:d\rho>0.
\end{equation}
Observe that the condition $\tilde{e}(\rho_\pm)=\frac{1}{2}\rho_\pm$ requires $\xi_2(\rho_+)=1$, $\eta_2(\rho_-)=-1$.

Inspired by the known families of subsolutions for the Muskat problem \cite{Sz-Muskat} or the Kelvin-Helmholtz instability \cite{Sz-KH}, it is of interest to investigate the limit case when one is in the boundary of the convex hull, instead of its interior, as this corresponds to the limiting mixing zone growth rates of these families. In our case this means to choose $|\xi|=|\eta|=1$ throughout all of $[\rho_-,\rho_+]$, i.e. $\xi_2\equiv-\eta_2\equiv 1$.
%, \todo{as this would allow for more mixing, at least on an intuitive level. ??} 
%
Of course this will not lead to a strict subsolution inside the mixing zone, so we will later consider a slight perturbation in order to be into the interior of the convex hull.
%\todo{more blabla? cont of energy, minimization of Q, maximality of mixing? etc}

Denote by $Q_0$, $G_0$, $\tilde e_0$ the functions associated with the choice $\xi_2\equiv-\eta_2\equiv 1$, i.e.
\begin{gather*}
Q_0(\rho)=(\rho-\rho_-)\sqrt{\rho_-}+(\rho_+-\rho)\sqrt{\rho_+},\quad \tilde{e}_0(\rho)=\frac{1}{2}\frac{\rho_+\rho_-(\rho_+-\rho_-)^2}{Q_0(\rho)^2},\\
G_0(\rho)=\frac{(\rho_+-\rho)(\rho-\rho_-)(\sqrt{\rho_-}-\sqrt{\rho_+})}{Q_0(\rho)}.
\end{gather*}
Lengthy, but straightforward computations show that $G_0$ is uniformly strictly convex on $[\rho_-,\rho_+]$ and also that $I_{1,-1}=0$. This means that with this choice there holds equality in \eqref{eq:weakadm} for any $t>0$.

We now turn to the perturbation. Let $\varepsilon>0$ and consider
\begin{align}
\xi_2(\rho):=1+\varepsilon \bar\xi(\rho),\quad \eta_2(\rho):=-1+\varepsilon \bar\eta(\rho),
\end{align}
with 
$\bar{\xi}<0$, $\bar{\eta}>0$ on $(\rho_-,\rho_+)$
and
$\bar{\xi}(\rho_\pm)=\bar{\eta}(\rho_\pm)=0$. Again, the last condition allows the function $e$ defined via \eqref{eq:subsolen} to be continuous over the whole domain $\Omega\times (0,T)$.

We will look for asymptotic expansions of the associated $Q=Q_\varepsilon$, $G=G_\varepsilon$, $\tilde e=\tilde e_\varepsilon$ with respect to $\varepsilon>0$. There holds
\begin{align*}
Q_\varepsilon(\rho)&=Q_0(\rho)+\varepsilon\left((\rho-\rho_-)\sqrt{\rho_-}\bar{\xi}-(\rho_+-\rho)\sqrt{\rho_+}\bar\eta \right)
=:Q_0(\rho)+\varepsilon\bar{Q}(\rho),\\
\tilde e_\varepsilon(\rho)&=\frac{1}{2}\frac{\rho_+\rho_-(\rho_+-\rho_-)^2}{(Q_0(\rho)+\varepsilon\bar Q(\rho))^2}=\tilde e_0(\rho)-\varepsilon\rho_+\rho_-(\rho_+-\rho_-)^2\frac{\bar{Q}(\rho)}{Q_0(\rho)^3}+\mathcal{O}(\varepsilon^2)\\
&=:\tilde e_0(\rho)+\varepsilon\bar{e}(\rho)+\mathcal{O}(\varepsilon^2),\\
G_\varepsilon(\rho)&=G_0(\rho)+\varepsilon \frac{(\rho_+-\rho)(\rho-\rho_-)}{Q_0^2(\rho)}\sqrt{\rho_+\rho_-}(\rho_+-\rho_-)(\bar\xi+\bar\eta)+\mathcal{O}(\varepsilon^2)\\
&=:G_0(\rho)+\varepsilon\bar{G}(\rho)+\mathcal{O}(\varepsilon^2),
\end{align*}
while the expansion of $I_\varepsilon:=I_{1+\varepsilon\bar{\xi},-1+\varepsilon\bar{\eta}}$ reads
\begin{align*}
I_\varepsilon=\varepsilon\int_{\rho_-}^{\rho_+} \left(\tilde{e}_0'(\rho)\bar G'(\rho)+\bar{e}'(\rho) G_0'(\rho)-\frac{3}{2} G_0'(\rho)\bar G'(\rho)\right) \, d\rho +\mathcal{O}(\varepsilon^2)=:\varepsilon \bar{I}+\mathcal{O}(\varepsilon^2).
\end{align*}
Since $G_0$ is uniformly convex on $[\rho_-,\rho_+]$, the perturbed function $G_\varepsilon$ will also be uniformly convex for small enough $\varepsilon>0$. Moreover, in order to have admissibility for $\varepsilon>0$ small enough, it suffices to have $\bar I>0$.

By integration by parts we rewrite
\begin{align*}
\bar I&=-\int_{\rho_-}^{\rho_+}\left( \tilde{e}_0''(\rho)\bar G(\rho)+\bar{e}(\rho) G_0''(\rho)-\frac{3}{2} G_0''(\rho)\bar G(\rho)\right) \, d\rho\\&=\int_{\rho_-}^{\rho_+}\bar\xi(\rho) H_1(\rho) \, d\rho+\int_{\rho_-}^{\rho_+}\bar\eta(\rho) H_2(\rho) \, d\rho,
\end{align*}
where
\begin{align*}
H_1(\rho)&=\frac{(\rho_+-\rho)(\rho-\rho_-)}{Q_0^2(\rho)}\sqrt{\rho_+\rho_-}(\rho_+-\rho_-)\left(\frac{3}{2}G''_0(\rho)-\tilde e''_0(\rho)\right)\\&\hspace{40pt}+\frac{\rho_+\rho_-(\rho_+-\rho_-)^2}{Q_0^3(\rho)}\sqrt{\rho_-}(\rho-\rho_-)G''_0(\rho),\\
H_2(\rho)&=\frac{(\rho_+-\rho)(\rho-\rho_-)}{Q_0^2(\rho)}\sqrt{\rho_+\rho_-}(\rho_+-\rho_-)\left(\frac{3}{2}G''_0(\rho)-\tilde e''_0(\rho)\right)\\&\hspace{40pt}-\frac{\rho_+\rho_-(\rho_+-\rho_-)^2}{Q_0^3(\rho)}\sqrt{\rho_+}(\rho_+-\rho)G''_0(\rho).
\end{align*}
It then follows that in order to have $\bar{I}>0$, it suffices to find $\bar\rho\in(\rho_-,\rho_+)$ such that either $H_1(\bar\rho)<0$ or $H_2(\bar\rho)>0$. Indeed, if $H_1(\bar\rho)<0$, one may choose a smooth function $\rho\mapsto\bar\xi(\rho)$  such that it is strictly negative on $(\rho_-,\rho_+)$, vanishes at the endpoints and concentrates at $\bar\rho$ sufficiently such that $\int_{\rho_-}^{\rho_+}\bar\xi(\rho) H_1(\rho) \, d\rho>0$. Then, regardless of the sign of $H_2$, one may clearly choose a function $\rho\mapsto\bar\eta=\bar\eta(\rho)$ which is strictly positive on $(\rho_-,\rho_+)$, vanishes at the endpoints, and is small enough such that $\bar{I}>0$. The case $H_2(\bar\rho)>0$ can be treated similarly.
 
Finally, to conclude the proof of Theorem \ref{thm:rayleigh_taylor_subsolutions}, we will prove that in fact the first case  $H_1(\bar\rho)<0$ is not possible, while $H_2(\bar\rho)>0$ is possible if and only if $\sqrt{\frac{\rho_+}{\rho_-}}>\frac{4+2\sqrt{10}}{3}$.

Let us first prove the second statement. $H_2(\bar\rho)>0$ is equivalent to
$$Q_0(\bar\rho)(\bar\rho-\rho_-)\sqrt{\rho_-}\left(\frac{3}{2}G''_0(\bar\rho)-\tilde e''_0(\bar\rho)\right)-\rho_+\rho_-(\rho_+-\rho_-)G''_0(\bar\rho)>0.$$
Plugging in the expressions for $Q_0,G_0$ and $\tilde{e}_0$, one obtains that this is equivalent to
$$\bar\rho^2-(\rho_++2\rho_-)\bar\rho+\frac{2}{3}\rho_+^{3/2}\rho_-^{1/2}+\frac{5}{3}\rho_+\rho_-+\rho_-^2<0.$$
This is possible only if the discriminant with respect to $\bar\rho$ is strictly positive, which reads
\begin{align*}
(\rho_++2\rho_-)^2-4\left(\frac{2}{3}\rho_+^{3/2}\rho_-^{1/2}+\frac{5}{3}\rho_+\rho_-+\rho_-^2\right)>0
\end{align*}
or equivalently 
\[
r^2\left(r^2-\frac{8}{3}r-\frac{8}{3}\right)=r^2\left(r-\frac{4-2\sqrt{10}}{3}\right)\left(r-\frac{4+2\sqrt{10}}{3}\right)>0,
\]
where we have denoted $r:=\sqrt{\frac{\rho_+}{\rho_-}}>1$. The statement then follows by taking for instance $\bar{\rho}=\frac{\rho_++2\rho_-}{2}\in(\rho_-,\rho_+)$ due to $\sqrt{\frac{\rho_+}{\rho_-}}>\frac{4+2\sqrt{10}}{3}$.

The case $H_1(\bar\rho)<0$ being not possible is proven similarly, the same calculations yield the condition $\frac{1}{r^2}-\frac{8}{3r}-\frac{8}{3}>0$, which is not possible for $r>1$.

It remains to compute the precise growth rates of the mixing zone $\mathscr{U}$ given in \eqref{eq:entreqq}.
Observe that $\partial_\xi G(\rho_\pm)=\partial_\eta G(\rho_\pm)=0$, such that $\bar{\xi}(\rho_\pm)=\bar{\eta}(\rho_\pm)=0$ implies
%so if in addition one also imposes the condition $\bar{\xi}'(\rho_\pm)=\bar{\eta}'(\rho_\pm)=0$, then one gets
\begin{align*}
G'_\varepsilon(\rho_\pm)&=G'_0(\rho_\pm)=\frac{\sqrt{\rho_\pm}-\sqrt{\rho_\mp}}{\sqrt{\rho_\mp}}.
\end{align*}

This concludes the proof of Theorem \ref{thm:rayleigh_taylor_subsolutions}.
\end{proof}

We would like to point out that the condition $\sqrt{\frac{\rho_+}{\rho_-}}>\frac{4+2\sqrt{10}}{3}$ only enters in the admissibility of the subsolutions, more precisely it comes from our construction above for assuring $\bar I>0$. For an arbitrary ratio $\frac{\rho_+}{\rho_-}>1$ the fact that in the unperturbed case $I_{-1,1}=0$ shows that there exist infinitely many turbulently mixing solutions with the exact same growth rates $c_\pm(t)$ violating the weak admissibility by an arbitrary small amount of energy. 

Furthermore, we summarize the other ansatzes used during our construction and note that they can all be seen as not too restrictive for different reasons:
%Finally, we note that the only other ansatz-s we have made during the construction can be seen as not too restrictive.
\begin{itemize}
\item The independence of $x_1$ can be interpreted as an averaging in the $x_1$ direction. 
\item $v\equiv 0$ for the subsolution is in harmony with the vanishing initial velocity and the fact that the subsolution corresponds to an averaging of solutions.
\item $\xi$ and $\eta$ only depending on $\rho$ allow us to find the density $\rho$ as the unique entropy solution of a relatively simple conservation law, this generalizes the construction from \cite{Sz-Muskat,Sz-KH}, where the unique viscosity solution of a Burgers equation was considered. In fact a similar conservation law also appeared in the relaxation of the two-phase porous media flow with different mobilities by Otto \cite{Otto}. Our intuition behind choosing $\xi$ and $\eta$ to be perturbations of $\pm e_2$ has been explained during the proof. Nonetheless, it would be interesting to see if other choices of $\xi$ and $\eta$ also lead to admissible subsolutions. 
\item The continuity of $e$ across $\partial\mathscr U$ is not a huge jump from Definition \ref{def:rt_subsolutions}, which combined with $v\equiv 0$ already implied that
$e=\frac{1}{2}g^2t^2\rho_+$ in $\{x_2>0\}\cap\mathscr{D}\setminus\overline{\mathscr{U}}$, respectively
$e=\frac{1}{2}g^2t^2\rho_-$ in $\{x_2<0\}\cap\mathscr{D}\setminus\overline{\mathscr{U}}$,
and therefore the continuity of $e$ in each of the three pieces $\{x_2<0\}\cap\mathscr{D}\setminus\overline{\mathscr{U}}$, $\mathscr U$ and $\{x_2>0\}\cap\mathscr{D}\setminus\overline{\mathscr{U}}$.
\end{itemize}

Finally, we would like to state further properties than those of the growth rates of the unperturbed ``subsolution'' associated with $\xi_2\equiv 1$, $\eta_2\equiv -1$ in an explicit way. Inversion of the derivative $G_0':[\rho_-,\rho_+]\rightarrow \left[-\frac{\sqrt{\rho_+}-\sqrt{\rho_-}}{\sqrt{\rho_+}},\frac{\sqrt{\rho_+}-\sqrt{\rho_-}}{\sqrt{\rho_-}}\right]$ shows that the density profile, defined in \eqref{eq:entropy_solution}, inside the mixing zone is given by
\begin{align*}
\rho\left(x_2,t\right)=\rho_++\sqrt{\rho_+\rho_-}+\rho_--\frac{(\sqrt{\rho_+}+\sqrt{\rho_-})\sqrt[4]{\rho_+\rho_-}}{\sqrt{1+\frac{2x_2}{gt^2}}},
\end{align*}
the relaxed momentum $u_2(x_2,t)=gtG_0(\rho(x_2,t))$ and $e$ defined in \eqref{eq:subsolen} inside $\mathscr{U}$ by
\begin{gather*}
u_2(x_2,t)=gt(\sqrt{\rho_+}+\sqrt{\rho_-})\left(\frac{\sqrt[4]{\rho_+\rho_-}}{\sqrt{1+\frac{2x_2}{gt^2}}}+\sqrt[4]{\rho_+\rho_-}\sqrt{1+\frac{2x_2}{gt^2}}-\sqrt{\rho_+}-\sqrt{\rho_-}\right)\\
e(x_2,t)=\frac{1}{2}g^2t^2\sqrt{\rho_-\rho_+}\left(1+\frac{2x_2}{gt^2}\right),
\end{gather*}
from which an interested reader can obtain a formula of the associated energy density $E_{sub}$ defined in \eqref{eq:subsolenergy}.
Here we would only like to state the conversion rate of total potential energy into total kinetic energy. Recall that the unperturbed ``subsolution'' satisfies \eqref{eq:weakadm} with equality. Hence the total kinetic energy at time $t\geq 0$ can be expressed as the difference in total potential energy, which is
\begin{align*}
\int_{\Omega}(\rho_0(x)-\rho(x,t))gx_2\:dx&=\frac{g^3t^4}{8}\int_{\rho_-}^{\rho_+}(\rho_0(G_0'(\rho))-\rho)\big(G'_0(\rho)^2\big)'\:d\rho\\
&=\frac{g^3t^4}{8}\int_{\rho_-}^{\rho_+}G'_0(\rho)^2\:d\rho\\
&=\frac{g^3t^4(\sqrt{\rho_+}+\sqrt{\rho_-})(\sqrt{\rho_+}-\sqrt{\rho_-})^3}{24\sqrt{\rho_+\rho_-}}.
\end{align*}
We conclude the paper by presenting a plot of the above density (blue) and momentum (red) profiles for the choice $\rho_-=1/4$, $\rho_+=4$, $g=1$ at fixed time $t=\left(\frac{2}{g(\sqrt{\rho_+}-\sqrt{\rho_-})}\right)^{\frac{1}{2}}$. At this specific time the mixing zone extends from $x_2=-\rho_+^{-1/2}=-1/2$ to $x_2=\rho_-^{-1/2}=2$.

\begin{figure}[ht]
\centering
\includegraphics[width=0.62\textwidth]{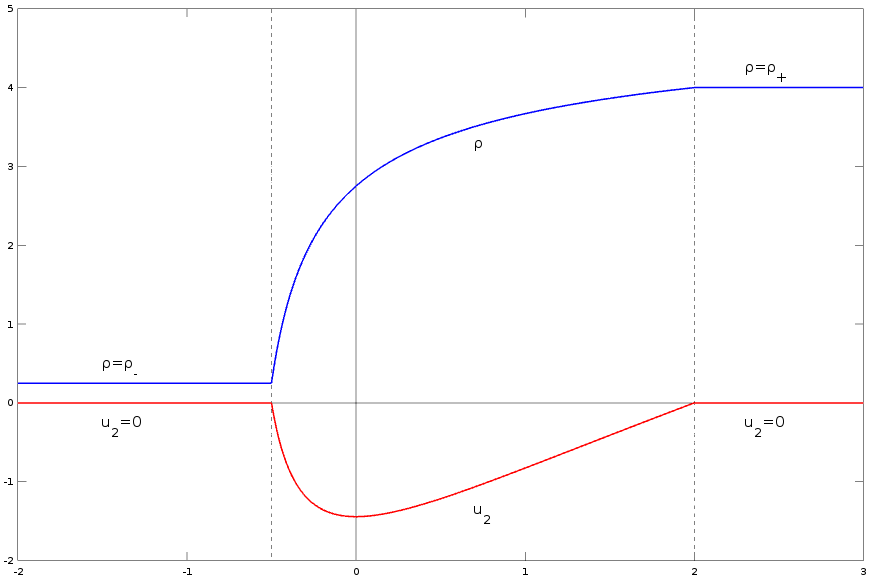}
\end{figure}

\vspace{20pt}
\noindent Mathematisches Institut,  Universit\"at Leipzig,  Augustusplatz 10, D-04109 Leipzig \\
\texttt{bjoern.gebhard@math.uni-leipzig.de}\\
\texttt{jozsef.kolumban@math.uni-leipzig.de}\\
\texttt{laszlo.szekelyhidi@math.uni-leipzig.de}

\end{document}